\documentclass[a4paper,10pt,reqno]{amsart}

\usepackage[a4paper,left=2.5cm, right=2.5cm, top=2.7cm, bottom=2.7cm,%
includehead,includefoot]{geometry}

\usepackage{amsmath,amssymb,amsfonts,amsthm}
\usepackage{graphicx}
\usepackage{subfigure}
\usepackage{float}
\usepackage{xcolor}
\usepackage[numbers,square,sort&compress]{natbib}
\usepackage{hyperref}
\hypersetup{colorlinks,citecolor=blue,linkcolor=blue,breaklinks=true}
\usepackage{booktabs}
\usepackage{colortbl}
\usepackage{caption}
\usepackage{enumitem}
\usepackage{epstopdf}
\usepackage{algorithm}
\usepackage{algpseudocode}
\usepackage{array}
\usepackage{bbding}
\usepackage{fancyhdr} 
\usepackage{fancyvrb} 
\usepackage{longtable}
\usepackage{listings}
\usepackage{rotating,rotfloat} 
\usepackage{yhmath} 
\usepackage{latexsym}
\usepackage{soul}

\usepackage{fancyhdr}
\allowdisplaybreaks

\newtheorem{theorem}{Theorem}[section]
\newtheorem{lemma}{Lemma}[section]

\newtheorem{proposition}[theorem]{Proposition}
\newtheorem{definition}{Definition}[section]

\newtheorem{remark}{Remark}
\newtheorem{convention}{Convention}

\newcommand{\bbm}{\begin{bmatrix}}
	\newcommand{\ebm}{\end{bmatrix}}

\begin{document}

	\title{Equivariant Index Theorem on $\mathbb{R}^n$ in the Context of Continuous Fields of $C^*$-algebras}
	
	\author[Baiying Ren]{Baiying Ren}
	\address[B. Ren]{\normalfont{Research Center for Operator Algebras, School of Mathematical Sciences\\
			East China Normal University\\ 3663 North Zhongshan Rd, Shanghai 200062, P.R.China.}}
	\email{52275500020@stu.ecnu.edu.cn}
	
	\author{Hang Wang}
	
	\address[H. Wang]{\normalfont{Research Center for Operator Algebras, School of Mathematical Sciences\\
			East China Normal University\\ 3663 North Zhongshan Rd, Shanghai 200062, P.R.China.}}
	\email{wanghang@math.ecnu.edu.cn}
	
	\author{Zijing Wang}
	
	\address[Z. Wang]{\normalfont{Research Center for Operator Algebras, School of Mathematical Sciences\\
			East China Normal University\\ 3663 North Zhongshan Rd, Shanghai 200062, P.R.China.}}
	\email{52205500014@stu.ecnu.edu.cn}
	
	\let\thefootnote\relax
	\footnotetext{MSC2020: Primary 58B34, Secondary 19K56, 58J20, 46L80, 19D55.}
	\keywords{Equivariant index theorem, continuous fields of $C^*$-algebras, cyclic cohomology, Bott-Dirac operator.}

	\date{}
	
	\maketitle
	
	\begin{abstract}
		We prove an equivariant index theorem on the Euclidean space using a continuous field of $C^*$-algebras. This generalizes the work of Elliott, Natsume and Nest, which is a special case of the algebraic index theorem by Nest-Tsygan. Using our formula, the equivariant index of the Bott-Dirac operator on $\mathbb{R}^{2n}$ can be explicitly calculated.
	\end{abstract}
    
    \section{Introduction}
    
    In \cite{Nest,Nest2}, Nest and Tsygan proved the algebraic index theorem which is parallel to the  Atiyah-Singer index theorem but in a purely algebraic context. They studied the class of elliptic pseudodifferential operators on compact smooth manifolds in \cite{Nest} and studied families of such operators in \cite{Nest2}. In \cite{ENN2}, Elliott, Natsume and Nest studied elliptic pseudodifferential operators on $\mathbb{R}^n$ of positive order in the sense of Shubin \cite{Sh}, which is a particular but an essential example of the algebraic index theorem. The general idea behind the proof is to regard the elements of $C_0(T^*\mathbb{R}^n)$ as the classical limit of elements of $\mathcal{K}\left(L^2(\mathbb{R}^n)\right)$, the algebra of compact operators on $L^2(\mathbb{R}^n)$, using the technique of continuous fields of $C^*$-algebras. In the following, let us recall the work of \cite{ENN2}.

    Let $P_a:\mathcal{S}(\mathbb{R}^n;E) \mapsto \mathcal{S}(\mathbb{R}^n;F)$ be a pseudodifferential operator (cf. Section \ref{s5}, abbreviated $\Psi$DO) of symbol $a$. 
    Here $E \to {\mathbb{R}^{n}}$ and $F \to {\mathbb{R}^{n}}$ are complex vector bundles and $\mathcal{S}(\mathbb{R}^n;E)$ denotes the space of rapidly decreasing sections of $E$, similarly for $\mathcal{S}(\mathbb{R}^n;F)$.
    Since $\mathbb{R}^{n}$ is contractible, $E$ and $F$ are trival bundles. Thus, there exist complex vector spaces $V,W$ such that
    $E\cong \mathbb{R}^{n}\times V$, 
    $F\cong \mathbb{R}^{n}\times W$.
    Then the operator can be written as  
    $P_a: \mathcal{S}(\mathbb{R}^{n};V) \to \mathcal{S}(\mathbb{R}^{n};W)$ with symbol $a\in C^{\infty}(T^*{\mathbb{R}}^n; \mathrm{End}(V,W))$, where $\mathcal{S}(\mathbb{R}^{n};V)$ denotes  the space of rapidly decreasing $V$-valued functions while $\mathcal{S}(\mathbb{R}^{n};W)$ is similar. 
    Moreover, if $P_a$ is elliptic (cf. Definition \ref{def4}), then $P_a: L^2(\mathbb{R}^{n};V) \to L^2(\mathbb{R}^{n};W)$ in the sense of Remark \ref{rmk1} is a Fredholm operator. 
    Denote by $\mathrm{ind}(P_a)$ the index of the Fredholm operator $P_a$. 
    By Remark \ref{rmk1}, the formal adjoint $P^*_a$ of $P_a$ is also elliptic, and the index is given by 
    \begin{align*}
    	\mathrm{ind}(P_a)
    	=\mathrm{dim}(\mathrm{Ker}P_a|_{L^2(\mathbb{R}^{n};V)})-\mathrm{dim}(\mathrm{Ker}P^*_a|_{L^2(\mathbb{R}^{n};W)}). 
    \end{align*}

    Elliott et al. obtained the following theorem in \cite{ENN2} describing the analytic index $\mathrm{ind}(P_a)$ in terms of $a$. The graph projection 
    $$e_a={\left(\begin{matrix}
    		(1+a^*a)^{-1} & (1+a^*a)^{-1}a^*\\
    		a(1+a^*a)^{-1} & a(1+a^*a)^{-1}a^*\\
    	\end{matrix}
    	\right)}$$ 
    induced by the closed multiplication operator $a: L^2(T^*{\mathbb{R}}^n; V) \to L^2(T^*{\mathbb{R}}^n; W)$ is an essential ingredient of the index formula:

    \begin{theorem}[\cite{ENN2}; Section 3]
    	\label{th1}
    	Suppose that $P_a$ is an elliptic pseudodifferential operator associtated to a symbol $a$ of positive order (cf. Definition \ref{def5}) on $\mathbb{R}^n$, then
    	\begin{align*}
    		\mathrm{ind}(P_a)= \frac{1}{(2\pi i)^nn!}\int_{T^*\mathbb{R}^n}^{} \mathrm{tr}(\hat{e_a}(d\hat{e_a})^{2n}),
    	\end{align*}
    	where $T^*\mathbb{R}^n$ is oriented by $dx_1\wedge d{\xi}_1 \wedge \dots \wedge dx_n\wedge d{\xi}_n $, $\hat{e_a}=e_a-\left(\begin{matrix}
    		0 & 0\\
    		0 & 1\\
    	\end{matrix}
    	\right)$ and the convergence of the integral on the right hand side follows from the fact that $a$ has positive order.
    \end{theorem}
    
    In order to state our result, let us recall the main idea of \cite{ENN2}. The index $\mathrm{ind}(P_a)$ can be regarded as an element of $K_0(\mathcal{K}(L^2(\mathbb{R}^n)))$ represented by the class of the graph projection $e$ of the closed operator associated with $P_a$. Note that the canonical trace $\mathrm{Tr}$ on $\mathcal{K}(L^2(\mathbb{R}^n))$ induces an isomorphism between $K_0(\mathcal{K}(L^2(\mathbb{R}^n)))$ and $\mathbb{Z}$. Then one can construct a continuous field of graph projections $(e_{\hslash})_{\hslash\in [0,1]}$ with specified restrictions satisfying $e_0=e_a$, $e_1=e$.
    Since the pairing of the class $[e_{\hslash}]\in K_0(\mathcal{K}(L^2(\mathbb{R}^n)))$ with the canonical trace $\mathrm{Tr}$ diverges at $\hslash=0$, we replace $\mathrm{Tr}$ by a densely defined equivalent (in cyclic periodic cohomology) cyclic cocycle $\omega$. The index formula is then obtained by taking the limit of the pairing as $\hslash$ converges to 0:
    \begin{align*}
    	\mathrm{ind}(P_a)\equiv \langle [e_{\hslash}], [\omega] \rangle_{\hslash>0} =\langle [e_{a}], [\epsilon] \rangle,
    \end{align*}
    where the last pairing is between some densely defined cyclic cocycle $\epsilon$ on $C_0(T^*\mathbb{R}^n)$ and the class $[e_a]$ in $K_0(C_0(T^*\mathbb{R}^n))$.

    In this paper, we study the equivariant version of Theorem \ref{th1}
    by considering a compact subgroup $G\leq SO(n)$ acting on $\mathbb{R}^n$ by isometry.
    {Let $P_a: \mathcal{S}(\mathbb{R}^{n};V) \to \mathcal{S}(\mathbb{R}^{n};W)$ be an elliptic pseudodifferential operator of positive order on $\mathbb{R}^n$}, where $V,W$ are complex vector spaces. Suppose $P_a$ is $G$-invariant, that is, $V,W$ are $G$-spaces and for all $g\in G$,
    \begin{align*}
    g|_{L^2({\mathbb{R}}^n; W)}P_a=P_ag|_{L^2({\mathbb{R}}^n; V)},
    \end{align*}
    where $g\in G$ is identified with the corresponding unitary operator on $L^2(\mathbb{R}^n;V)$ (resp. $L^2(\mathbb{R}^n;W)$), denoted by $g|_{L^2({\mathbb{R}}^n; V)}$ (resp. $g|_{L^2({\mathbb{R}}^n; W)}$).
	Denote by $\mathrm{ind}_G(P_a)$ the equivariant index of a $G$-invariant elliptic pseudodifferential operator $P_a$ given by,
	\begin{align}
		\label{map2}
		\mathrm{ind}_G(P_a):\  G &\to \mathbb{C}\\
		g &\mapsto  
		\begin{aligned}[t]
			\mathrm{ind}_{(g)}(P_a)&=\mathrm{Tr}(g|_{L^2({\mathbb{R}}^n; V)}\mathrm{Ker}P_a)-\mathrm{Tr}(g|_{L^2({\mathbb{R}}^n; W)}\mathrm{Ker}P_a^*),
		\end{aligned} \notag
	\end{align}
    where $\mathrm{Ker}P_a$ denotes the orthogonal projection of ${L^2({\mathbb{R}}^n; V)}$ onto the kernel of $P_a$, while $\mathrm{Ker}P_a^*$ is similar.

	The main result of this paper is the generalization of Theorem \ref{th1} to $G$-invariant operators, stated as follows.
	It is easy to see that Theorem \ref{th3} reduces to Theorem \ref{th1} when $g$ is the group identity.
	\begin{theorem}
		\label{th3}
		Suppose $G \le SO(n)$ is a compact subgroup and $P_a: \mathcal{S}(\mathbb{R}^{n};V) \to \mathcal{S}(\mathbb{R}^{n};W)$ is a $G$-invariant elliptic pseudodifferential operator of positive order on $\mathbb{R}^n$, where $V,W$ are $G$-spaces. For $g\in G$, if ${(\mathbb{R}^n)}^g$, the fixed-point set of the action of $g$ on $\mathbb{R}^n$, has positive dimension, then
		\begin{align*}
			\mathrm{ind}_{(g)}(P_a)= \frac{1}{(2\pi i)^{n_g}{n_g}!\mathrm{det}(g-1) }\int_{T^*{(\mathbb{R}^n)}^g} \mathrm{tr}\left[\left(\begin{matrix}
				g^V & 0\\
				0 & g^W\\
			\end{matrix}
			\right) \hat{e_a}(d\hat{e_a})^{2n_g}\right],
		\end{align*}
		where $n_g=\mathrm{dim}{(\mathbb{R}^n)}^g$, $g^V$ (resp. $g^W$) denotes the matrix corresponding to $g\in Aut(V)$ (resp. $Aut(W)$).
	\end{theorem}

    \begin{theorem}
    	\label{th4}
    	Suppose $G \le SO(n)$ is a compact subgroup and $P_a: \mathcal{S}(\mathbb{R}^{n};V) \to \mathcal{S}(\mathbb{R}^{n};W)$ is a $G$-invariant elliptic pseudodifferential operator of positive order on $\mathbb{R}^n$, where $V,W$ are $G$-spaces.
    	If ${(\mathbb{R}^n)}^g$, the fixed-point set of the action of $g\in G$ on $\mathbb{R}^n$, has only isolated points, that is, ${(\mathbb{R}^n)}^g=\{0\}$.
    	Then
    	\begin{align*}
    		\mathrm{ind}_{(g)}(P_a)= \frac{1}{\mathrm{det}(g-1) } \mathrm{tr}\left[\left(\begin{matrix}
    			g^V & 0\\
    			0 & g^W\\
    		\end{matrix}
    		\right) \hat{e_a}(0,0)\right],
    	\end{align*}
    	where $g^V$ (resp. $g^W$) denotes the matrix corresponding to $g\in Aut(V)$ (resp. $Aut(W)$).
    \end{theorem}

    Refer to the following remark for the notation $\mathrm{det}(g-1)$ in Theorem \ref{th3} and Theorem \ref{th4}.
    
   \begin{remark}
   	  It is clear that $(\mathbb{R}^n)^g$ is a subspace, for every $g\in G$. Decompose $\mathbb{R}^n$ as
   	  \begin{align*}
   	  	\mathbb{R}^n= (\mathbb{R}^n)^g \oplus \mathcal{N}(\mathbb{R}^n)^g,\ \forall g\in G,
   	  \end{align*}
     where $\mathcal{N}(\mathbb{R}^n)^g$ denotes the orthogonal complement subspace of $(\mathbb{R}^n)^g$. Denote by $\mathrm{det}(g-1)$ the determinant of the linear transform $g-1$ on $\mathcal{N}(\mathbb{R}^n)^g$.
   \end{remark}
    
     The main challenge of this paper is as follows.
     Note that the continuous field of graph projections  $(e_{\hslash})_{\hslash\in [0,1]}$ exists as above. 
     However, one has to construct the analogues of the cyclic cocycles $\omega$ and $\epsilon$ in the equivariant case.
     On one hand, to compute the equivariant index of the $G$-invariant operator $P_a$,
     $\omega$ has to be replaced by a family of cyclic cocycles $(\omega_g)_{g\in G}$ satisfying 
      \begin{align*}
     	\mathrm{ind}_{(g)}(P_a) = \langle [e_{1}], [{\omega}_g] \rangle.
     \end{align*}
     On the other hand, after rescaling, one has to show that when $\hslash$ converges to 0,
     \begin{align*}
     	\lim_{\hslash \to 0}\langle [e_{\hslash}], [{\omega}_g] \rangle =\langle [e_{a}], [\epsilon_g] \rangle,
     \end{align*}
     where $\epsilon$ is replaced by a family of cyclic cocycles $(\epsilon_g)_{g\in G}$ such that the right hand side of the above equality is the fixed point formula.

    Our main motivation is to verify that the equivariant index of the Bott-Dirac operator on $\mathbb{R}^{2n}$ is equal to 1, which sheds light on the $\gamma$-element associated with the Baum-Connes conjecture for isometry groups of $\mathbb{R}^{2n}$.

    The paper is organized as follows. In Section 2, we introduce the class of pseudodifferential operators studied by M. Shubin in \cite{Sh}. Section 3 is divided into four parts to prove Theorem \ref{th3} and Theorem \ref{th4} following the idea of \cite{ENN2}. In Section 3.1, $\mathrm{ind}_G(P_a)$ is represented by the class of a $G$-invariant graph projection $e$ in $K^G_0(\mathcal{K}(L^2(\mathbb{R}^n;V\oplus W)))$.
    In Section 3.2-3.3, we construct a continuous field of $G$-invariant graph projections $(e_{\hslash})_{\hslash\in [0,1]}$ with specified restrictions $e_0=e_a$, $e_1=e$.
    In Section 3.4, we generalize $\omega$ and $\epsilon$ in \cite{ENN2} to families of cyclic cocycles $(\omega_g)_{g\in G}$ and $(\epsilon_g)_{g\in G}$ and show that the equivariant analytic index is equal to the pairing:
    \begin{align*}
    	\mathrm{ind}_{(g)}(P_a)\equiv \langle [e_{\hslash}], [{\omega}_g] \rangle_{\hslash>0} =\langle [e_{a}], [\epsilon_g] \rangle.
    \end{align*}
    The last formula of the above equality is obtained by taking the limit of the pairing as $\hslash$ converges to 0.
    In Section 4, we give an example on the equivariant index of the Bott-Dirac operator on $\mathbb{R}^{2n}$.

    Throughout this paper, for a $G$-$C^{*}$-algebra $A$, i.e., a $C^{*}$-algebra $A$ with a norm continuous action of a locally compact and separable group $G$, we will use the notation $A^G$ for the $C^{*}$-subalgebra of the $G$-invariant elements of $A$. For a non-unital $C^{*}$-algebra $A$, denote by $A^{\sim}$ the unitization of $A$.

	\section{Preliminaries}
	\label{s5}

	Here we summarize certain facts on the class of $\Psi$DO we are interested in. See \cite{Sh} for more details.
	
	Let $u,v\in \mathcal{S}({{\mathbb{R}}^n})$ be rapidly decreasing functions. The Fourier transform of $u$ is defined by 
	\begin{align*}
		(Fu)(\xi)=\int_{{\mathbb{R}^{n}}}{e^{-i\langle x, \xi\rangle}u(x)}dx,
	\end{align*}
	where $\langle x, \xi\rangle$ denotes the Euclidean inner product of ${\mathbb{R}}^n$ between $x$ and $\xi$, and the inverse Fourier transform of $v$ is given by 
	\begin{align*}
		(F^{-1}v)(x)=(2\pi)^{-n}\int_{{\mathbb{R}^{n}}}{e^{i\langle x, \xi\rangle}v(\xi)}d\xi.
	\end{align*}
	Then $u$ can be written as 
	\begin{align*}
		u(x)=(2\pi)^{-n}\int_{{\mathbb{R}^{n}}}{e^{i\langle x, \xi\rangle}(Fu)(\xi)}d{\xi}.
	\end{align*}
	Let $a\in {C^{\infty}(T^*{\mathbb{R}}^n)}$. A pseudodifferential operator on ${\mathbb{R}}^n$ associated to $a$ is an operator of the form
	\begin{align}
		\label{eq1}
		(P_au)(x)=(2\pi)^{-n}\int_{{\mathbb{R}^{n}}}{e^{i\langle x, \xi\rangle}a(x,\xi)(Fu)(\xi)}d{\xi},
	\end{align}
	where $u\in \mathcal{S}({{\mathbb{R}}^n})$ and $a$ will be called the (total) symbol of $P_a$. In order for the above formula to make sense, the symbol needs to satisfy further conditions in the following definition.
	\begin{definition}[\cite{Sh}]
		\label{def5}
		Let $m\in \mathbb{R}$. The symbol class $\Gamma^m({{\mathbb{R}}^n})$ consists of functions $a\in {C^{\infty}({\mathbb{R}}^n)}$ such that for any multi-index $\alpha$, there exists a constant $C_{\alpha}$ such that
		\begin{align*}
			\left|{\partial}^{\alpha}_z(a)\right| \le C_{\alpha}(1+|z|^2)^{\frac{m-|{\alpha}|}{2}},
		\end{align*} 
	where $|z|^2={z_1}^2+\cdots+{z_n}^2$ for $z=(z_1,\cdots,z_n)\in {\mathbb{R}}^n$. Elements of $\Gamma^m({{\mathbb{R}}^n})$ are said to be of order $m$.
	\end{definition} 
	A symbol $a\in \Gamma^m({{\mathbb{R}}^n\times {\mathbb{R}}^n})$ defines a $\Psi$DO  $P_a$ by formula (\ref{eq1}). The operator $P_a$ is also said to be of order $m$. It can be checked that 
	\begin{align*}
		P_a:\mathcal{S}(\mathbb{R}^n) \to \mathcal{S}(\mathbb{R}^n)
	\end{align*} 
    is a continuous operator.
	
	In this paper, matrix-valued symbols are considered. 
	Let $V={\mathbb{C}}^k$ and $W={\mathbb{C}}^l$, where $k,l\in \mathbb{N}$. 
	Let $a\in M_{l\times k}{(\Gamma^m({\mathbb{R}}^n\times{{\mathbb{R}}^n}))}$ be an $l\times k$ matrix of symbols of order $m$. Then the formula (\ref{eq1}) still makes sense after replacing $u\in \mathcal{S}({{\mathbb{R}}^n})$ by $\mathcal{S}({{\mathbb{R}}^n}; V)$. Thus 
	\begin{align*}
		P_a:\mathcal{S}(\mathbb{R}^n;V) \to \mathcal{S}(\mathbb{R}^n;W)
	\end{align*}
    is obtained.
    We shall say that $a$ is an $M_{l\times k}$-valued symbol of order $m$.
    In this paper, we focus on the 
    square matrix valued symbols when $V=W={\mathbb{C}}^k$ (but we still distinguish $V$ and $W$ because spaces of the same dimension may differ as $G$-spaces in the equivariant version) and investigate the ellipticity in the following definition.
	\begin{definition}
		\label{def4}
		Let $a$ be an $M_{k}$-valued symbol of order $m$. Then $a$ is said to be elliptic if there exist positive constants $C$ and $R$ such that
		\begin{align*}
			a(x,\xi)^*a(x,\xi) \ge C(|x|^2+|\xi|^2)^mI_k\ \mathrm{for}\  |x|^2+|\xi|^2\ge R,
		\end{align*}
	where $I_k$ denotes the $k\times k$ identity matrix, $\ge$ refers to the usual ordering of self-adjoint matrices and $*$ denotes the usual adjoint of matrices.
	\end{definition}
	If $a$ is an elliptic symbol, then the $\Psi$DO  $P_a$ by formula (\ref{eq1}) is also said to be elliptic.

	For each $a\in \Gamma^m({{\mathbb{R}}^n\times {\mathbb{R}}^n})$, there exists $a^{\prime}\in \Gamma^m({{\mathbb{R}}^n\times {\mathbb{R}}^n})$ such that
	\begin{align*}
		(P_a\phi, \psi)=(\phi, P_{a^{\prime}}\psi),\ \phi, \psi\in \mathcal{S}({\mathbb{R}}^n),
	\end{align*}
	where $(\cdot,\cdot)$ refers to the inner product in $L^2({\mathbb{R}}^n)$. Thus $P_a$ is formally adjoint to $P_{a^{\prime}}$. Since $P_{a^{\prime}}$ is also continuous, $P_a$ extends to a continuous map 
	\begin{align*}
		P_a:\mathcal{S}(\mathbb{R}^n)^{\prime} \to \mathcal{S}(\mathbb{R}^n)^{\prime},
	\end{align*}
	where $\mathcal{S}(\mathbb{R}^n)^{\prime}$ denotes the space of tempered distributions.
	
	\begin{remark}
		\begin{enumerate}
			\item [(1)] By Theorem 24.3 of \cite{Sh}, if $a\in \Gamma^0({{\mathbb{R}}^n\times {\mathbb{R}}^n})$, then $P_a:L^2(\mathbb{R}^n) \to L^2(\mathbb{R}^n)$ is a bounded operator. 
			\item [(2)] By Theorem 24.4 of \cite{Sh}, if $a\in \Gamma^m({{\mathbb{R}}^n\times {\mathbb{R}}^n})$ and $m < 0$, then $P_a:L^2(\mathbb{R}^n) \to L^2(\mathbb{R}^n)$ is a compact operator.
		\end{enumerate}
	\end{remark}
	
	\begin{remark}
		\label{rmk1}
		If $P_a:\mathcal{S}(\mathbb{R}^n;V) \to \mathcal{S}(\mathbb{R}^n;W)$ is an elliptic $\Psi$DO, by Theorem 25.3 of \cite{Sh}, then $P_a$ is a Fredholm operator. Since $P_a$ is elliptic, the formal adjoint $P^*_a$ of $P_a$ is also elliptic, and the space $\mathrm{Im}(P_a)$ in $\mathcal{S}(\mathbb{R}^{n};W)$ is the orthogonal complement to $\mathrm{Ker}P^*_a|_{\mathcal{S}(\mathbb{R}^{n};W)}$ with respect to the inner product in $L^2({\mathbb{R}}^n; W)$. Then the Fredholm index is obtained by 
		\begin{align*}
			\mathrm{ind}(P_a)
			=\mathrm{dim}(\mathrm{Ker}P_a|_{\mathcal{S}(\mathbb{R}^{n};V)})-\mathrm{dim}(\mathrm{Ker}P^*_a|_{\mathcal{S}(\mathbb{R}^{n};W)}). 
		\end{align*}
		Consider 
		\begin{align}
			\label{s4}
			P_a: L^2(\mathbb{R}^n;V)\to L^2(\mathbb{R}^n;W)
		\end{align} 
	    as a possibly unbounded operator with dense domain $\{f\in L^2(\mathbb{R}^n;V)| P_af\in L^2(\mathbb{R}^n;W) \}$. Then by Theorem 25.3 of \cite{Sh}, $P_a$ of the form (\ref{s4}) is also a Fredholm operator. It can be verified that indices of the two Fredholm operators are equal. Indeed, by Theorem 25.2 of \cite{Sh},
		\begin{align*}
			\mathrm{Ker}P_a|_{\mathcal{S}(\mathbb{R}^{n};V)}=\mathrm{Ker}P_a|_{L^2(\mathbb{R}^{n};V)},
		\end{align*}
	and while $P^*_a$ is similar, we obtain
	\begin{align*}
		\mathrm{ind}(P_a)
		=\mathrm{dim}(\mathrm{Ker}P_a|_{L^2(\mathbb{R}^{n};V)})-\mathrm{dim}(\mathrm{Ker}P^*_a|_{L^2(\mathbb{R}^{n};W)}). 
	\end{align*}
	\end{remark}


	\section{Proof of the Equivariant Index Theorem}

	\subsection{The Equivariant Analytic Index}
	
	Let $P: \mathcal{S}({\mathbb{R}}^n; V) \to \mathcal{S}({\mathbb{R}}^n; W)$ be an equivariant elliptic $\Psi$DO of order $m>0$. 
	Consider $P$ as an unbounded operator on $L^2({\mathbb{R}}^n; V)$ with dense domain $\mathcal{S}({\mathbb{R}}^n; V)$. 
	Since $P$ has the formal adjoint $P^*$, $P$ is closable.  
	Denote by $T$ the closure of $P$. It can be verified that $T$ is $G$-invariant.
	
	Let $e$ denote the orthogonal projection of $L^2({\mathbb{R}}^n; V)\oplus L^2({\mathbb{R}}^n; W) \cong L^2({\mathbb{R}}^n; V\oplus W) $ onto the graph of the closed operator $T$. The projection can be represented by the matrix:
	$$
			e={\left(\begin{matrix}
					(1+T^*T)^{-1} & (1+T^*T)^{-1}T^*\\
					T(1+T^*T)^{-1} & T(1+T^*T)^{-1}T^*\\
				\end{matrix}
				\right)}.
	$$
	We call $e$ the graph projection of $T$. 
	By Section 4 of \cite{ENN2} and the fact that $T$ is $G$-invariant, we have
    \begin{align*}
    	e-\left(\begin{matrix}
    		0 & 0\\
    		0 & 1\\
    	\end{matrix}
    	\right)\in (\mathcal{K}(L^2({\mathbb{R}}^n; V\oplus W)))^G.
    \end{align*}
    
    Let $\mathrm{Ker}P$ (resp. $\mathrm{Ker}P^*$) denote the projection of $L^2({\mathbb{R}}^n; V)$ (resp. $L^2({\mathbb{R}}^n; W)$) onto the kernel of $P$ (resp. $P^*$), respectively.
    \begin{proposition}
    	\label{prop7}
    	$[e]-\left[\left(\begin{matrix}
    		0 & 0\\
    		0 & 1\\
    	\end{matrix}
    	\right)\right]= [\mathrm{Ker}P]-[\mathrm{Ker}P^*]$ in $K_0\left((\mathcal{K}(L^2({\mathbb{R}}^n; V\oplus W)))^G\right)$.
    \end{proposition}
    Here by abuse of notation, $\mathrm{Ker}P$ (resp. $\mathrm{Ker}P^*$) also denotes the projection of $L^2({\mathbb{R}}^n; V\oplus W)$ induced by the inclusion map.
    \begin{proof}
    	The homotopy argument in the proof of Theorem 4.1 in \cite{ENN2} still holds in the $G$-invariant version.
    \end{proof}

	\subsection{Continuous Field of $C^*$-algebras}
	\label{s1}
	
	The idea of the continuous field of $C^*$-algebras is based on \cite{ENN2}. Consider the $(2n+1)$-dimensional Heisenberg group $H_{2n+1}=\{[v,x,t]| v,x\in \mathbb{R}^n,t\in \mathbb{R}
	\}$, where we denote $[v,x,t]$ by
	\begin{align*}
		\left(\begin{matrix}
			1 & x & t\\
			\ & \ddots & v\\
			0 & \ & 1
		\end{matrix}\right).
	\end{align*}
    Then we shall explain that the group $C^*$-algebra $C^*(H_{2n+1})$ can be identified as the algebra of continuous sections vanishing at infinity of a continuous field of $C^*$-algebras $(A^{\prime}_{\hslash})_{\hslash \in \mathbb{R}}$, where
	\begin{align*}
		A^{\prime}_{0}&=C_0(T^*{\mathbb{R}}^n),\\
		A^{\prime}_{\hslash}&=\mathcal{K}(L^2({\mathbb{R}}^n)),\ \hslash\neq 0.
	\end{align*}
    Here we may restrict the field to the interval $[0,1]$.
    
    Actually $H_{2n+1}$ can be viewed as a semidirect product. Denote by two closed subgroups of $H_{2n+1}$ as follows:
    \begin{align*}
    	M&=\{[0,x,t]|x\in \mathbb{R}^n,t\in \mathbb{R} \}\cong \mathbb{R}^{n+1},\\
    	N&=\{[v,0,0]|v\in \mathbb{R}^n \}\cong \mathbb{R}^n.
    \end{align*}
    Then $H_{2n+1}=M\rtimes_{\tau} N=\mathbb{R}^{n+1}\rtimes_{\tau} \mathbb{R}^n$, where
    \begin{align*}
    	(\tau{[v,0,0]})[0,x,t]=[v,0,0]^{-1}[0,x,t][v,0,0]=[0,x-tv,t],
    \end{align*}
    and it follows that
    \begin{align*}
    	C^*(H_{2n+1})=C_0(\mathbb{R}\times \mathbb{R}^n) \rtimes_{\tau} \mathbb{R}^n,
    \end{align*}
    where 
    \begin{align}
    	\label{eq2}
    	(\tau (v)f)(t,x)=f(t,x-tv),\ \forall v\in \mathbb{R}^n,\ \forall f\in C_0(\mathbb{R}\times \mathbb{R}^n),\ \forall t\in \mathbb{R},\ \forall x\in \mathbb{R}^n.
    \end{align}
    From the identification $C_0(\mathbb{R}\times \mathbb{R}^n)\cong C_0(\mathbb{R}, C_0(\mathbb{R}^n))$, $C_0(\mathbb{R}\times \mathbb{R}^n)$ can be viewed as a continuous field of $C^*$-algebras over $\mathbb{R}$ with fiber $C_0( \mathbb{R}^n)$, while by (\ref{eq2}), the action of $\mathbb{R}^n$ preserves fibers. Thus it is clear that $C^*(H_{2n+1})$ is a continuous field of $C^*$-algebras over $\mathbb{R}$ parametrised by $t$ with fiber $C_0( \mathbb{R}^n)\rtimes_{\tau_t} \mathbb{R}^n$ identified as follows.
    At $t=0$, when $\tau$ is reduced to the trival action, the fiber is $C_0( \mathbb{R}^n)\otimes C^*(\mathbb{R}^n)\cong C_0(T^*\mathbb{R}^n)$; when $t\neq 0$, by Takai duality, the fiber is $\mathcal{K}(L^2({\mathbb{R}}^n))$. 
    
    The continuous field structure is given by identifying $f\in \mathcal{S}(\mathbb{R}^{2n})$ within $C_0( \mathbb{R}^n)\rtimes_{\tau_t} \mathbb{R}^n$ and representing the latter elements on $L^2(\mathbb{R}^n,C_0( \mathbb{R}^n))$ with $t$ varying in $\mathbb{R}$, denoted by section
    \begin{align*}
    	t \mapsto \pi_t(f).
    \end{align*} 
    By direct computation, for $g\in L^2(\mathbb{R}^n,C_0( \mathbb{R}^n))$ and $s\in \mathbb{R}^n$,
    \begin{align*}
    	\pi_t(f)g(s)&=\int_{\mathbb{R}^n} f(\cdot,v) (v.g)(s)dv\\
    	&=\int_{\mathbb{R}^n} f(\cdot,v) (\tau_t(v)(g(s-v)))dv,
    \end{align*}
    after identifying $g$ with a function on $\mathbb{R}^n\times \mathbb{R}^n$, and by evaluating both sides of the above equality at $x\in \mathbb{R}^n$,
    \begin{align*}
    	\pi_t(f)g(x,s)=\int_{\mathbb{R}^n} f(x,v) g(x-tv,s-v)dv.
    \end{align*}
    For $t>0$, denote by $\rho_t(f)$ the compact operator on $L^2(\mathbb{R}^n)$ of the form
    \begin{align}
    	\label{eq3}
    	\rho_t(f)g(x)=\int_{{\mathbb{R}^{n}}}f(x,v)g(x-tv)dv.
    \end{align}
    In fact, the two representations $\pi_t$ and $\rho_t\otimes1$ are unitarily equivalent,
    since
    \begin{align*}
    	(U^*_t\pi_t(f)U_tg)(x,s)=\int_{{\mathbb{R}^{n}}}f(x,v)g(x-tv,s)ds,
    \end{align*}
    where
    \begin{align*}
    	(U_tg)(x,s)=g(x,s-\frac{x}{t}).
    \end{align*}
    For $t=0$, the identification $C_0( \mathbb{R}^n)\otimes C^*(\mathbb{R}^n)\cong C_0(T^*\mathbb{R}^n)$ implies that $\pi_0(f)=f$.
    
    In this paper, we make use of the continuous field structure in \cite{ENN2}, given by the family of sections of the form
    \begin{align*}
    	\Gamma=\{\hslash\in [0,1] \mapsto \rho_{\hslash}(\hat{f})\in A^{\prime}_{\hslash}|f\in \mathcal{S}(\mathbb{R}^{2n+1})
    	\},
    \end{align*}
    where $\hslash \in [0,1]$, $f\in \mathcal{S}(\mathbb{R}^{2n+1})$ and $\hat{f}$ denotes the Fourier transform with respect to the second variable of $f$, i.e.,
    \begin{align*}
    	\hat{f}(x,y,h)=\int_{{\mathbb{R}^{n}}} f(x,\xi, h)e^{-i\langle y, \xi\rangle}d\xi, \ \forall x,y\in \mathbb{R}^{n}, \ \forall h\in \mathbb{R}.
    \end{align*}
    Here for $\hslash>0$, $\rho_{\hslash}$ is expressed in $\hat{f}$ as below and is  slightly different from (\ref{eq3}),
    \begin{align*}
    	\rho_{\hslash}(\hat{f}) \phi(x)=\frac{1}{{(2\pi)}^{n}}\int_{{\mathbb{R}^{n}}} \hat{f}(x,y,\hslash)\phi(x+\hslash y)dy,
    \end{align*}
    which implies that $\rho_{\hslash}(\hat{f})$ is compact on $L^2(\mathbb{R}^n)$.
    For $\hslash=0$, define $\rho_{0}(\hat{f})\in C_0(T^*\mathbb{R}^n)$ as the restriction of $f$, i.e., $\rho_{0}(\hat{f})(x,\xi)=f(x,\xi,0)$.

	By \cite{ENN1} (Theorem 2.4 and proof of Theorem 3.1) and \cite{ENN2}, the families 
	$(A_{\hslash})_{\hslash \in [0,1]}$ and $({A_{\hslash}}^{\sim})_{\hslash \in [0,1]}$ are also continuous fields, where
	\begin{align*}
		A_{0}&=C_0(T^*{\mathbb{R}}^n;\mathrm{End}(V\oplus W)),\\
		A_{\hslash}&=\mathcal{K}(L^2({\mathbb{R}}^n;V\oplus W)),\ \hslash>0.
	\end{align*}
	
	\subsection{Equivariant Continuous Vector Field}
	\label{s2}
	Let $P:\mathcal{S}(\mathbb{R}^n;V) \mapsto \mathcal{S}(\mathbb{R}^n;W)$ be an equivariant elliptic pseudodifferential operator of order $m>0$. Denote by $a$ the symbol of operator $P$.
	We construct an equivariant continuous vector field of $({A_{\hslash}}^{\sim})_{\hslash \in [0,1]}$.
	For $\hslash >0$, let $P_{\hslash}$ be the elliptic operator with given symbol $a_{\hslash}(x,\xi):=a(x,\hslash \xi)$ and denote by $e_{\hslash}$ the graph projection of $P_{\hslash}$. Then $e_{\hslash}  \in \mathcal{K}(L^2({\mathbb{R}}^n; V\oplus W))^{\sim}$.
	
	For $\hslash =0$, consider the symbol $a$ as an $M_k$-valued $C^{\infty}$-function on $T^*{\mathbb{R}^n}$ and regard $a$ as the pointwise multiplication operator from $L^2(T^*{\mathbb{R}^n};V)$ to $L^2(T^*{\mathbb{R}^n};W)$. Denote by $e_a$ the graph projection of operator $a$. Since $a$ is the symbol of order $m>0$, by Section 3 of \cite{ENN2}, $e_a \in C_0(T^*{\mathbb{R}}^n;\mathrm{End}(V\oplus W))^{\sim}$. Set $e_0=e_a$. 
	\begin{theorem}
		\label{th2}
		The vector field $e=(e_{\hslash})_{\hslash \in [0,1]}$ of $({A_{\hslash}}^{\sim})_{\hslash \in [0,1]}$ is continuous.
	\end{theorem}
	\begin{proof}
		Here we omit the proof. For details see proof of Theorem 5.2 of \cite{ENN2}.
	\end{proof}

	Denote by $\mathcal{K}^{\infty}$ (resp. $\mathcal{K}^{\infty}(L^2({\mathbb{R}}^n;V\oplus W))$) the subalgebra of integral operators with Schwartz kernels in  $\mathcal{S}({\mathbb{R}^n} \times {\mathbb{R}^n})$ (resp. $\mathcal{S}({\mathbb{R}^n} \times {\mathbb{R}^n};\mathrm{End}(V\oplus W))$).
	Then for $\hslash>0$,  $\{\rho_{\hslash}(\hat{f}) |f\in \mathcal{S}(\mathbb{R}^{2n+1};\mathrm{End}(V\oplus W))\}=\mathcal{K}^{\infty}(L^2({\mathbb{R}}^n;V\oplus W))$. In the case of $\hslash = 0$, $\{\rho_{0}(\hat{f}) |f\in \mathcal{S}(\mathbb{R}^{2n+1};\mathrm{End}(V\oplus W))\}=\mathcal{S}(T^*{\mathbb{R}}^n;\mathrm{End}(V\oplus W))$.

	Denote by ${(A_{\hslash}}^{\infty})_{\hslash \in [0,1]}$ the sub-family of $({A_{\hslash}})_{\hslash \in [0,1]}$, where
	\begin{align*}
		A_{0}^{\infty}&=\mathcal{S}(T^*{\mathbb{R}}^n;\mathrm{End}(V\oplus W)),\\
		A_{\hslash}^{\infty}&=\mathcal{K}^{\infty}(L^2({\mathbb{R}}^n;V\oplus W)),\ \hslash>0.
	\end{align*}
	Then ${(A_{\hslash}}^{\infty})_{\hslash \in [0,1]}$ is also a continuous field with continuous field structure $\Gamma$, which is given by
    \begin{align*}
    	\Gamma=\{\hslash \mapsto \rho_{\hslash}(\hat{f})\in A_{\hslash}^{\infty}|f\in \mathcal{S}({\mathbb{R}}^{2n+1};\mathrm{End}(V\oplus W))
    	\}.
    \end{align*}
    Furthermore, ${(A_{\hslash}}^{\infty})_{\hslash \in [0,1]}$ is the smooth subalgebra $\mathcal{S}^*(H_{2n+1})$ of $C^*(H_{2n+1})$. Since $\mathcal{S}^*(H_{2n+1})$ is closed under the holomorphic functional calculus in $C^*(H_{2n+1})$, for a given $0<\epsilon<1$, by Theorem \ref{th2}, there exists a continuous field $e^{\infty}=(e_{\hslash}^{\infty})_{\hslash \in [0,1]}\in \Gamma^{\sim}$ of projections such that 
    \begin{align*}
    	\left\| {e-e^{\infty}}\right\|=\sup\limits_{\hslash\in [0,1]}\left\| {e_{\hslash}-e^{\infty}_{\hslash}}\right\|_{\hslash}  <\epsilon,
    \end{align*}
    where $\left\|\cdot \right\|$ refers to sup-norm and $\left\| \cdot \right\|_{\hslash}$ denotes the norm of $A_{\hslash}^{\sim}$.
    Thus $e_{\hslash}$ is homotopic to $e^{\infty}_{\hslash}$ in $A_{\hslash}^{\sim}$ for $\hslash \in [0,1]$.
	Moreover, since $e=(e_{\hslash})_{\hslash \in [0,1]}$ is $G$-invariant, we can also choose $e^{\infty}$ to be $G$-invariant such that $e_{\hslash}$ is homotopic to $e^{\infty}_{\hslash}$ through a path of $G$-invariant projections in $A_{\hslash}^{\sim}$ for $\hslash \in [0,1]$, which implies that
	\begin{align*}
		[e_{\hslash}]-\left[\left(\begin{matrix}
			0 & 0\\
			0 & 1\\
		\end{matrix}
		\right)\right] = [e^{\infty}_{\hslash}]-\left[\left(\begin{matrix}
			0 & 0\\
			0 & 1\\
		\end{matrix}
		\right)\right] \in K_0\left(A_{\hslash}^G\right).
	\end{align*}
    In fact, suppose that \begin{align*}
    	e^{\infty}-\left(\begin{matrix}
    		0 & 0\\
    		0 & 1\\
    	\end{matrix}
    	\right)=\left(e_{\hslash}^{\infty}-\left(\begin{matrix}
    		0 & 0\\
    		0 & 1\\
    	\end{matrix}
    	\right)\right)_{\hslash \in [0,1]}=(\rho_{\hslash}(\hat{f_0}))_{\hslash \in [0,1]}
    \end{align*} for an $f_0\in \mathcal{S}({\mathbb{R}}^{2n+1};\mathrm{End}(V\oplus W))$, then by changing into \begin{align*}
    e^{\infty}-\left(\begin{matrix}
    	0 & 0\\
    	0 & 1\\
    \end{matrix}
    \right)=\left(e_{\hslash}^{\infty}-\left(\begin{matrix}
    	0 & 0\\
    	0 & 1\\
    \end{matrix}
    \right)\right)_{\hslash \in [0,1]}=(\rho_{\hslash}(\hat{f}))_{\hslash \in [0,1]},
    \end{align*} 
where
	\begin{align*}
		f=\int_G g.f_0dg,
	\end{align*}
	we obtain for all $\hslash\in [0,1]$, 
	\begin{align}
		\label{s10}
		\left\|e_{\hslash}-\left(\begin{matrix}
			0 & 0\\
			0 & 1\\
		\end{matrix}
		\right)-\rho_{\hslash}(\hat{f}) \right\|_{\hslash}&=\left\|e_{\hslash}-\left(\begin{matrix}
			0 & 0\\
			0 & 1\\
		\end{matrix}
		\right)-\int_G\rho_{\hslash}(\widehat{g.f_0})dg \right\|_{\hslash}.
	\end{align}
    We claim that $\rho_{\hslash}(\widehat{g.f_0})=g.\rho_{\hslash}(\widehat{f_0})$, then 
    \begin{align*}
		(\ref{s10})&=\left\|e_{\hslash}-\left(\begin{matrix}
			0 & 0\\
			0 & 1\\
		\end{matrix}
		\right)-\int_G g.\rho_{\hslash}(\widehat{f_0})dg \right\|_{\hslash}\\
		&=\left\|\int_G g.\left(e_{\hslash}-\left(\begin{matrix}
			0 & 0\\
			0 & 1\\
		\end{matrix}
		\right)- \rho_{\hslash}(\widehat{f_0})\right)dg \right\|_{\hslash}\\
		&\leq \left\| {e-e^{\infty}}\right\|  <\epsilon.
	\end{align*}
	As for the claim that $\forall k\in  SO(2n)$, $\forall f \in \mathcal{S}(T^*{\mathbb{R}}^n \times {\mathbb{R}})$, 
	\begin{align*}
		k.\rho_{\hslash}(\hat{f})=k|_{L^2(\mathbb{R}^n)}\rho_{\hslash}(\hat{f})k^{-1}|_{L^2(\mathbb{R}^n)}=\rho_{\hslash}(\widehat{k.f}),
	\end{align*}
	by direct computation, $\forall \phi \in L^2(\mathbb{R}^n)$, $\forall x\in \mathbb{R}^n$, we obtain
	\begin{align*}
		k|_{L^2(\mathbb{R}^n)}\rho_{\hslash}(\hat{f})k^{-1}|_{L^2(\mathbb{R}^n)}\phi(x)&=\frac{1}{{(2\pi)}^{n}}\int_{{\mathbb{R}^{n}}}\hat{f}(k^{-1}x,y,\hslash)(k^{-1}.\phi)(k^{-1}x+\hslash y)dy\\
		&=\frac{1}{{(2\pi)}^{n}}\int_{{\mathbb{R}^{n}}}\hat{f}(k^{-1}x,y,\hslash)\phi(x+\hslash ky)dy\\
		&=\frac{1}{{(2\pi)}^{n}}\int_{{\mathbb{R}^{n}}}\hat{f}(k^{-1}x,k^{-1}\overline{y},\hslash)\phi(x+\hslash \overline{y})d\overline{y}\\
		&=\frac{1}{{(2\pi)}^{n}}\int_{{\mathbb{R}^{n}}}\int_{{\mathbb{R}^{n}}}{f}(k^{-1}x,\xi,\hslash)e^{-i\langle k^{-1}\overline{y} , \xi  \rangle}\phi(x+\hslash \overline{y})d\xi d\overline{y},\\
		&=\frac{1}{{(2\pi)}^{n}}\int_{{\mathbb{R}^{n}}}\int_{{\mathbb{R}^{n}}}{f}(k^{-1}x,k^{-1}\overline{\xi},\hslash)e^{-i\langle k^{-1}\overline{y} , k^{-1}\overline{\xi}  \rangle}\phi(x+\hslash \overline{y})d\overline{\xi} d\overline{y}\\
		&=\frac{1}{{(2\pi)}^{n}}\int_{{\mathbb{R}^{n}}}\int_{{\mathbb{R}^{n}}}{f}(k^{-1}x,k^{-1}\overline{\xi},\hslash)e^{-i\langle \overline{y} , \overline{\xi}  \rangle}\phi(x+\hslash \overline{y})d\overline{\xi} d\overline{y}\\
		&=\frac{1}{{(2\pi)}^{n}}\int_{{\mathbb{R}^{n}}} \widehat{k.f}(x,\overline{y},\hslash)\phi(x+\hslash \overline{y}) d\overline{y}\\
		&=\rho_{\hslash}(\widehat{k.f})\phi(x).
	\end{align*}

	\subsection{A Field of Cyclic Cocycles on Any Sufficiently Smooth Field of Projections}
	
	\label{3.4}
	
	Let $P$ be a $\Psi$DO under the conditions of Section \ref{s2}.
	It is known that $\mathcal{K}^{\infty} \subseteq \mathcal{L}_1\left({L^2({\mathbb{R}}^n)}\right)$. Denote by $\mathrm{Tr}$ the restriction of the canonical trace of $\mathcal{K}\left({L^2({\mathbb{R}}^n)}\right)$ to $\mathcal{K}^{\infty}$.
    We consider the operator $\partial_{x_j}=\frac{\partial}{\partial{x_j}}$ and by abuse of notation, the pointwise multiplication operator $x_j$ by the coordinate function $x_j$ on $L^2({\mathbb{R}}^n)$, for $j=1, \cdots,n$. Denote by the induced derivations on $\mathcal{K}^{\infty}$ as follows,
    \begin{align*}
    	\delta_{2j-1}: \mathcal{K}^{\infty} &\to \mathcal{K}^{\infty}\\
    	S &\mapsto \delta_{2j-1}(S)=[\partial_{x_j}, S],
    \end{align*}
    \begin{align*}
    	\delta_{2j}: \mathcal{K}^{\infty} &\to \mathcal{K}^{\infty}\\
    	S &\mapsto \delta_{2j}(S)=[x_j, S],
    \end{align*}
    where $j=1, \cdots,n$. Naturally, the above derivations can be generalized to $\mathcal{K}^{\infty}(L^2(\mathbb{R}^n;V\oplus W))$.
    
    \begin{definition}
    	\label{def1}
    	For a fixed $g \in G$, suppose that $\mathrm{dim}\left({\mathbb{R}^n}\right)^g=n_g\ge0$, define the cyclic $2n_g$-cocycle on $\left(\mathcal{K}^{\infty}\right)^g$ as follows,
    \begin{align*}
    	\omega_g(T_0,\cdots,T_{2n_g})=  \frac{(-1)^{n_g}}{n_g!}\sum_{\sigma \in S_{2n_g}}{\mathrm{sgn}(\sigma)\mathrm{Tr}\left({g|_{L^2({\mathbb{R}}^n)}}T_0 \delta_{\sigma(1)}(T_1) \cdots \delta_{\sigma(2n_g)}(T_{2n_g}) \right)},
    \end{align*}
        where
     \begin{align*}
     	\left(\mathcal{K}^{\infty}\right)^g=\{T\in \mathcal{K}^{\infty}|\  g|_{L^2({\mathbb{R}}^n)}T=Tg|_{L^2({\mathbb{R}}^n)} \},
     \end{align*}
    and 
    \begin{align*}
    	T_0, \cdots, T_{2n_g} \in \left(\mathcal{K}^{\infty}\right)^g.
    \end{align*}
    Naturally, the definition of $\omega_g$ can be generalized to a cyclic $2n_g$-cocycle on $\left(\mathcal{K}^{\infty}(L^2(\mathbb{R}^n;V\oplus W))\right)^g$ as follows,
    \begin{align}
    	\label{s3}
    	\omega_g(T_0,\cdots,T_{2n_g})=  \frac{(-1)^{n_g}}{n_g!}\sum_{\sigma \in S_{2n_g}}{\mathrm{sgn}(\sigma)\mathrm{Tr}\left(\hat{g}T_0 \delta_{\sigma(1)}(T_1) \cdots \delta_{\sigma(2n_g)}(T_{2n_g}) \right)},
    \end{align}
    where
    \begin{align*}
    	\hat{g}=\left(\begin{matrix}
    		g|_{L^2({\mathbb{R}}^n;V)} & 0\\
    		0 &g|_{L^2({\mathbb{R}}^n;W)}
    	\end{matrix}\right),
    \end{align*}
    \begin{align*}
    	\left(\mathcal{K}^{\infty}(L^2(\mathbb{R}^n;V\oplus W))\right)^g=\{T\in \mathcal{K}^{\infty}(L^2(\mathbb{R}^n;V\oplus W))|\  \hat{g}T=T\hat{g} \},
    \end{align*}
    and
    \begin{align*}
    	T_0, \cdots, T_{2n_g} \in \left(\mathcal{K}^{\infty}(L^2(\mathbb{R}^n;V\oplus W))\right)^g.
    \end{align*}
    Here when $\mathrm{dim}\left({\mathbb{R}^n}\right)^g=n_g=0$, $\omega_g$ reduces to 
    \begin{align*}
    	\omega_g(T)=\mathrm{Tr}\left({g|_{L^2({\mathbb{R}}^n)}}T \right), \ T\in \left(\mathcal{K}^{\infty}\right)^g,
    \end{align*}
    and
    \begin{align*}
    	\omega_g(T)=\mathrm{Tr}\left(\hat{g}T \right), \ T\in \left(\mathcal{K}^{\infty}(L^2(\mathbb{R}^n;V\oplus W))\right)^g.
    \end{align*}
    \end{definition}

    The cyclic cocycle property for $\omega_g$ will be checked in Propsition \ref{prop5}.

    \begin{convention}
    	For a fixed $g\in G$, since $(\mathbb{R}^n)^g$ is a subspace, in this paper, we shall consider $(\mathbb{R}^n)^g$ as the subspace generated by $x_1,\cdots, x_{n_g}$, the first $n_g$-dimensional coordinate bases.
    \end{convention}
    The above convention derives exactly the following proposition by direct computation, which is needed in Proposition \ref{prop5}.
    
     \begin{proposition}
    	\label{prop4}
    	The action of $g\in G$ on $L^2(\mathbb{R}^n)$ commutes with operators $x_i$ and $\partial_{x_i}$, $i=1,\cdots, n_g$, that is,
    	\begin{align*}
    		g|_{L^2({\mathbb{R}}^n)}x_i=x_ig|_{L^2({\mathbb{R}}^n)},\ g|_{L^2({\mathbb{R}}^n)}{\partial_{x_i}}={\partial_{x_i}}g|_{L^2({\mathbb{R}}^n)},\ i=1,\cdots, n_g.
    	\end{align*}
    \end{proposition}

    \begin{proposition}
    	\label{prop5}
    	For a fixed $g \in G$, $w_g$ defined in Definition \ref{def1} is a cyclic $2n_g$-cocycle on $\left(\mathcal{K}^{\infty}\right)^g$ (resp. $\left(\mathcal{K}^{\infty}(L^2(\mathbb{R}^n;V\oplus W))\right)^g$).
    \end{proposition}
    \begin{proof}
    	Without loss of generality, we just consider $\omega_g$ defined on $\left(\mathcal{K}^{\infty}\right)^g$. 
    	
    	For any $T_0, \cdots, T_{2n_g}, T_{2n_g+1} \in \left(\mathcal{K}^{\infty}\right)^g$, by direct computation, we obtain, 
    	\begin{equation*}
    		\begin{aligned}[t]
    			&b\omega_g(T_0,\cdots,T_{2n_g},T_{2n_g+1})\\
    			=&  \frac{(-1)^{n_g}}{n_g!}\sum_{\sigma \in S_{2n_g}}\mathrm{sgn}(\sigma)\left[{\mathrm{Tr}\left({g|_{L^2({\mathbb{R}}^n)}}T_0T_1 \delta_{\sigma(1)}(T_2) \cdots \delta_{\sigma(2n_g)}(T_{2n_g+1}) \right)}\right. \\ 
    			&\;+ \sum_{i=1}^{2n_g}{(-1)^i\mathrm{Tr}\left({g|_{L^2({\mathbb{R}}^n)}}T_0 \delta_{\sigma(1)}(T_1) \cdots\delta_{\sigma(i)}(T_iT_{i+1})\cdots \delta_{\sigma(2n_g)}(T_{2n_g}) \right)}\\
    			&\;-\left. {\mathrm{Tr}\left({g|_{L^2({\mathbb{R}}^n)}}T_{2n_g+1}T_0 \delta_{\sigma(1)}(T_1) \cdots \delta_{\sigma(2n_g)}(T_{2n_g}) \right)}\right]\\
    			=&\frac{(-1)^{n_g}}{n_g!}\sum_{\sigma \in S_{2n_g}}\mathrm{sgn}(\sigma)\left[\mathrm{Tr}\left({g|_{L^2({\mathbb{R}}^n)}}T_0 \delta_{\sigma(1)}(T_1) \cdots \delta_{\sigma(2n_g)}(T_{2n_g})T_{2n_g+1} \right) \right.\\ &-\left. {\mathrm{Tr}\left({g|_{L^2({\mathbb{R}}^n)}}T_{2n_g+1}T_0 \delta_{\sigma(1)}(T_1) \cdots \delta_{\sigma(2n_g)}(T_{2n_g}) \right)}\right]\\
    			=&0.
    		\end{aligned}
    	\end{equation*}
    Thus, $\omega_g$ is a $2n_g$-cocycle on $\left(\mathcal{K}^{\infty}\right)^g$. As for the cyclic property of $\omega_g$, for any $T_0, \cdots, T_{2n_g} \in \left(\mathcal{K}^{\infty}\right)^g$, we have,
    \begin{align}
    	\label{s6}
    	&\omega_g(T_0,T_{1}\cdots,T_{2n_g})-(-1)^{2n_g}\omega_g(T_{2n_g},T_0\cdots,T_{2n_g-1})\\
    	=&\frac{(-1)^{n_g}}{n_g!}\sum_{\sigma \in S_{2n_g}}{\mathrm{sgn}(\sigma)\mathrm{Tr}\left({g|_{L^2({\mathbb{R}}^n)}}T_0 \delta_{\sigma(1)}(T_1) \cdots \delta_{\sigma(2n_g)}(T_{2n_g}) \right)}\notag\\
    	&-\frac{(-1)^{n_g}}{n_g!}\sum_{\tau \in S_{2n_g}}{\mathrm{sgn}(\tau)\mathrm{Tr}\left({g|_{L^2({\mathbb{R}}^n)}}T_{2n_g} \delta_{\tau(1)}(T_0) \cdots \delta_{\tau(2n_g)}(T_{2n_g-1}) \right)}\notag,
    \end{align}
    then by Proposition \ref{prop4},
    \begin{align*}
    	(\ref{s6})=&\frac{(-1)^{n_g}}{n_g!}\sum_{\sigma \in S_{2n_g}}{\mathrm{sgn}(\sigma)\mathrm{Tr}\left({g|_{L^2({\mathbb{R}}^n)}}\delta_{\sigma(2n_g)}(T_{2n_g})T_0 \delta_{\sigma(1)}(T_1) \cdots \delta_{\sigma(2n_g-1)}(T_{2n_g-1}) \right)}\\
    	&+\frac{(-1)^{n_g}}{n_g!}\sum_{\sigma \in S_{2n_g}}{\mathrm{sgn}(\sigma)\mathrm{Tr}\left({g|_{L^2({\mathbb{R}}^n)}}T_{2n_g} \delta_{\sigma(2n_g)}(T_0)\delta_{\sigma(1)}(T_1) \cdots \delta_{\sigma(2n_g-1)}(T_{2n_g}-1) \right)}.
    \end{align*}
    By the fact that $\delta_{i}\delta_{j}=\delta_{j}\delta_{i}$, $j,k=1,\cdots, 2n_g$, we obtain,
    \begin{align*}
    	(\ref{s6})&=\frac{(-1)^{n_g}}{n_g!}\sum_{\sigma \in S_{2n_g}}{\mathrm{sgn}(\sigma)\mathrm{Tr}\left(\delta_{\sigma(2n_g)}\left({g|_{L^2({\mathbb{R}}^n)}}T_{2n_g}T_0 \delta_{\sigma(1)}(T_1) \cdots \delta_{\sigma(2n_g-1)}(T_{2n_g-1}) \right)\right)}\\
    	&=0.
    \end{align*}
    \end{proof}

    \begin{remark}
    	\label{remark1}
    	For a $2l$-cyclic cocyle $\phi$ on $C^*$-algebra $A$ and an idempotent $e\in M_k(A)$, where $M_k(A)$ denotes by $k\times k$ matrices over $A$, the pairing between the cyclic cohomology class $[\phi]$ and $[e] \in K_0(A)$ is given by, 
    	\begin{align*}
    		\langle [e], [\phi] \rangle = (2\pi i)^{-l}(l!)^{-1}\phi\#\mathrm{tr}(e,\cdots, e),
    	\end{align*}
        for more details see \cite{Connes}.
        From the above discussion, the pairing of $[\omega_g]$ with $K_0\left(\left(\mathcal{K}^{\infty}(L^2(\mathbb{R}^n;V\oplus W))\right)^g\right)$ is well defined. Since $\mathcal{K}^{\infty}(L^2(\mathbb{R}^n;V\oplus W))$ is dense in $\mathcal{K}(L^2(\mathbb{R}^n;V\oplus W))$ and stable under the holomorphic functional calculas, we also consider the pairing of $[\omega_g]$ with $K_0\left(\left(\mathcal{K}(L^2(\mathbb{R}^n;V\oplus W))\right)^g\right)$.
    	
    	
    	
    \end{remark}

	\begin{proposition}
		\label{prop1}
		\begin{itemize}
			\item [(1)] For any idempotent $T \in \left(\mathcal{K}^{\infty}\right)^g$, $\langle [T],[(2\pi i)^{n_g}(n_g)!\omega_g] \rangle=\mathrm{Tr}({g|_{L^2({\mathbb{R}}^n)}}T)$.
			\item [(2)]  For any idempotent $T \in \left(\mathcal{K}^{\infty}(L^2(\mathbb{R}^n;V\oplus W))\right)^g$, $\langle [T], [(2\pi i)^{n_g}(n_g)!\omega_g] \rangle=\mathrm{Tr}(\hat{g}T)$.
		\end{itemize}
	\end{proposition}
    
    \begin{proof}
    	Without loss of generality, we just prove the case of (1). 
    	For a fixed $g \in G$, observe that by the assumption that $T \in \left(\mathcal{K}^{\infty}\right)^g$ and Proposition \ref{prop4},
    	\begin{align*}
    		g|_{L^2({\mathbb{R}}^n)}T=Tg|_{L^2({\mathbb{R}}^n)},\ g|_{L^2({\mathbb{R}}^n)}\partial_{x_j}=\partial_{x_j}g|_{L^2({\mathbb{R}}^n)},\ g|_{L^2({\mathbb{R}}^n)}x_j=x_jg|_{L^2({\mathbb{R}}^n)},
    	\end{align*}
        for $j=1,\cdots,n_g$.
        Denote by a cyclic 2-cocycle
        \begin{align*}
        	\overline{\omega_g}(T_0,T_1,T_2)=-\sum_{\sigma \in S_{2}}{\mathrm{sgn}(\sigma)\mathrm{Tr}\left({g|_{L^2({\mathbb{R}}^n)}}T_0 \delta_{\sigma(1)}(T_1)  \delta_{\sigma(2)}(T_{2}) \right)},
        \end{align*}
        where
        \begin{align*}
        	T_0,T_1,T_2 \in \left(\mathcal{K}^{\infty}\right)^g.
        \end{align*}
        By direct computation,
        \begin{align*}
        	\overline{\omega_g}(T,T,T)=\mathrm{Tr}({g|_{L^2({\mathbb{R}}^n)}}T).
        \end{align*}
     Then similarly to the proof of Proposition \ref{prop6} (see Appendix), the conclusion of (1) can be derived by induction.
    \end{proof}

    Now we continue with the mapping (\ref{map2}) and turn to the equivariant index of $P$. Since $P$ is an elliptic $\Psi$DO of positive order, the kernel of $P$ in $L^2(\mathbb{R}^n; V)$ is generated by finitely many rapidly decreasing functions, which implies that $\mathrm{Ker}P \in \left(\mathcal{K}^{\infty}(L^2(\mathbb{R}^n;V))\right)^G$. Similarly, $\mathrm{Ker}P^*\in \left(\mathcal{K}^{\infty}(L^2(\mathbb{R}^n;W))\right)^G$. By Propsition \ref{prop1} and Propsition \ref{prop7}, we obtain 
    \begin{align}
    	\label{map4}
    	\mathrm{ind}_G(P):\  G &\to \mathbb{C}\\
    	g &\mapsto
    	\begin{aligned}[t]
    		\mathrm{ind}_{(g)}(P)&=\mathrm{Tr}(g|_{L^2({\mathbb{R}}^n; V)}\mathrm{Ker}P)-\mathrm{Tr}(g|_{L^2({\mathbb{R}}^n; W)}\mathrm{Ker}P^*)\\
    		&=\langle \left[\mathrm{Ker}P\right]-[\mathrm{Ker}P^*],[(2\pi i)^{n_g}(n_g)!\omega_g] \rangle\\
    		&=\langle \left[e_{1}\right]-\left[\left(\begin{matrix}
    			0 & 0\\
    			0 & 1\\
    		\end{matrix}
    		\right)\right],[(2\pi i)^{n_g}(n_g)!\omega_g] \rangle\\
    		&=\langle \left[e^{\infty}_{1}\right]-\left[\left(\begin{matrix}
    			0 & 0\\
    			0 & 1\\
    		\end{matrix}
    		\right)\right],[(2\pi i)^{n_g}(n_g)!\omega_g] \rangle.
    		\end{aligned}\notag
    	\end{align}
    Then as the map $\hslash>0 \mapsto e_{\hslash}^{\infty}$ is norm continuous, 
    \begin{align*}
    		\mathrm{ind}_{(g)}(P)=\langle [e^{\infty}_{\hslash}]-\left[\left(\begin{matrix}
    			0 & 0\\
    			0 & 1\\
    		\end{matrix}
    		\right)\right],[(2\pi i)^{n_g}(n_g)!\omega_g] \rangle\ \mathrm{for\  all\  \hslash>0}.
    \end{align*}

	Next we show that as $\hslash \to 0$, the pairing of the last equality of the mapping (\ref{map4}) converges to the pairing on $\left(\mathcal{S}(T^*{\mathbb{R}}^n;\mathrm{End}(V\oplus W))\right)^g$.
	To see that, we shall firstly define a $g$-twisted cyclic cocycle on it.
	\begin{definition}
		\label{def2}
		For a fixed $g \in G$, suppose that $\mathrm{dim}\left({\mathbb{R}^n}\right)^g=n_g>0$, define a cyclic $2n_g$-cocycle on $\left(\mathcal{S}(T^*{\mathbb{R}}^n)\right)^g$
		as follows,
		\begin{align*}
			\epsilon_g(f_0,\cdots,f_{2n_g})
			&=\frac{1}{(2\pi i)^{n_g}n_g!{\mathrm{det}(g-1)}}
			\int_{T^*(\mathbb{R}^n)^{g}}f_0 d{f_1}\cdots d{ f_{2n_g}},
		\end{align*}
	where
	\begin{align*}
		\left(\mathcal{S}(T^*{\mathbb{R}}^n)\right)^g=\{f\in \mathcal{S}(T^*{\mathbb{R}}^n)| (g.f)(x,\xi)
		=f(g^{-1}x,g^{-1}\xi)=f(x,\xi),\ \forall (x,\xi)\in T^*{\mathbb{R}}^n\},
	\end{align*}
   and
   \begin{align*}
   	f_0,\cdots,f_{2n_g} \in \left(\mathcal{S}(T^*{\mathbb{R}}^n)\right)^g.
   \end{align*}

    Naturally, the definition of $\omega_g$ can be generalized to a cyclic $2n_g$-cocycle on $\left(\mathcal{S}(T^*{\mathbb{R}}^n;\mathrm{End}(V\oplus W))\right)^g$ as follows,
    \begin{align*}
    	\epsilon_g(f_0,\cdots,f_{2n_g})
    	&=\frac{1}{(2\pi i)^{n_g}n_g!{\mathrm{det}(g-1)}}
    	\int_{T^*(\mathbb{R}^n)^{g}}\mathrm{tr}\left[\left(\begin{matrix}
    		g^V & 0\\
    		0 & g^W\\
    	\end{matrix} \right) \left(f_0 d{f_1}\cdots d{ f_{2n_g}} \right)\right],
    \end{align*}
    where
    \begin{align*}
    	\left(\mathcal{S}(T^*{\mathbb{R}}^n;\mathrm{End}(V\oplus W))\right)^g=\{f\in \mathcal{S}(T^*{\mathbb{R}}^n;\mathrm{End}(V\oplus W))| &(g.f)(x,\xi)\\
    	=&\left(\begin{matrix}
    		g^V & 0\\
    		0 & g^W\\
    	\end{matrix} \right)f(g^{-1}x,g^{-1}\xi)\left(\begin{matrix}
    	g^V & 0\\
    	0 & g^W\\
    \end{matrix} \right)^{-1}\\=&f(x,\xi),\ \forall (x,\xi)\in T^*{\mathbb{R}}^n \},
    \end{align*}
    and
    \begin{align*}
    	f_0,\cdots,f_{2n_g} \in \left(\mathcal{S}(T^*{\mathbb{R}}^n;\mathrm{End}(V\oplus W))\right)^g.
    \end{align*}
	\end{definition}
	
	\begin{definition}
		\label{def3}
		For a fixed $g\in G$,
		if $g$ has only isolated fixed points on $\mathbb{R}^n$, that is, ${(\mathbb{R}^n)}^g=\{0\}$.
		Then define a $g$-twisted trace on $\left(\mathcal{S}(T^*{\mathbb{R}}^n)\right)^g$
		as follows,
		\begin{align*}
			\epsilon_g(f)
			&=\frac{1}{(2\pi i)^{n_g}n_g!{\mathrm{det}(g-1)}}f(0,0),
		\end{align*}
		where $\left(\mathcal{S}(T^*{\mathbb{R}}^n)\right)^g$ is defined as above and $f \in \left(\mathcal{S}(T^*{\mathbb{R}}^n)\right)^g$.
		
	Naturally, the definition of $\omega_g$ can be generalized to a $g$-twisted trace on $\left(\mathcal{S}(T^*{\mathbb{R}}^n;\mathrm{End}(V\oplus W))\right)^g$ as follows,
	\begin{align*}
		\epsilon_g(f)
		&=\frac{1}{(2\pi i)^{n_g}n_g!{\mathrm{det}(g-1)}} \mathrm{tr}\left[\left(\begin{matrix}
			g^V & 0\\
			0 & g^W\\
		\end{matrix} \right) f(0,0)\right],
	\end{align*}
	where $\left(\mathcal{S}(T^*{\mathbb{R}}^n;\mathrm{End}(V\oplus W))\right)^g$ is defined as above and $f \in \left(\mathcal{S}(T^*{\mathbb{R}}^n;\mathrm{End}(V\oplus W))\right)^g$.
	\end{definition}

	\begin{proposition}
		For a fixed $g\in G$, $\epsilon_g$ defined in Definiton \ref{def2} is a cyclic $2n_g$-cocycle on $\left(\mathcal{S}(T^*{\mathbb{R}}^n;\mathrm{End}(V\oplus W))\right)^g$.	\end{proposition}
	\begin{proof}
		Based on Proposition \ref{prop4}, the conclusion can be checked by direct computation similar to the proof of Proposition \ref{prop5}.
	\end{proof}
	Refers to Remark \ref{remark1} for the pairing between the cyclic cohomology class $[\epsilon_g]$ and $K_0\left(\left(\mathcal{S}(T^*{\mathbb{R}}^n)\right)^g\right)\cong K_0\left(\left(C_0(T^*{\mathbb{R}}^n)\right)^g\right)$ (resp. $K_0\left(\left(\mathcal{S}(T^*{\mathbb{R}}^n;\mathrm{End}(V\oplus W))\right)^g\right)\cong K_0\left(\left(C_0(T^*{\mathbb{R}}^n;\mathrm{End}(V\oplus W))\right)^g\right)$)  .
	
	Before discussing the convergence of the last equality of the mapping (\ref{map4}), we also need the following observations, which are some kind of generalization of ones in \cite{ENN2}.

	For $f_1,f_2\in \mathcal{S}(T^*{\mathbb{R}}^n \times {\mathbb{R}};\mathrm{End}(V\oplus W))$, we define $\hat{f_1} \ast_H \hat{f_2} \in \mathcal{S}(T^*{\mathbb{R}}^n \times {\mathbb{R}};\mathrm{End}(V\oplus W))$ by
	\begin{align*}
		(\hat{f_1} \ast_H \hat{f_2})(x,y,\hslash)=\left\{\begin{aligned}
			&(2\pi)^{-n}\int_{\mathbb{R}^n}\widehat{f_1}(x,z,\hslash)\widehat{f_2}(x+\hslash z,y-z,\hslash)dz, &\  \hslash \neq 0,\\
			&\widehat{f_1f_2}(x,y,0), &\  \hslash=0,
		\end{aligned}
		\right.
	\end{align*}
	where $f_1f_2$ stands for the product of $f_1$ and $f_2$ in $\mathcal{S}(T^*{\mathbb{R}}^n;\mathrm{End}(V\oplus W))$, $x\in \mathbb{R}^n$, $y\in \mathbb{R}^n$ and $\hslash \in \mathbb{R}$.
	We consider the case when $\hslash \in [0,1]$.

	\begin{proposition}
		\label{prop3}
		\begin{itemize}
			\item [(1)] For $f\in \mathcal{S}(T^*{\mathbb{R}}^n \times {\mathbb{R}};\mathrm{End}(V\oplus W))$,
			\begin{align*}
				\delta_{2j-1}(\rho_{\hslash}(\hat{f}))&=[\partial_{x_j}, \rho_{\hslash}(\hat{f})]=[\partial_{x_j}, \rho_{\hslash}(\hat{f})]=\rho_{\hslash}\left(\widehat{\frac{\partial f}{\partial x_j}} \right),\ j=1,\cdots,n,\\
				\delta_{2j}(\rho_{\hslash}(\hat{f}))&=[x_j, \rho_{\hslash}(\hat{f})]=[x_j, \rho_{\hslash}(\hat{f})]=\left(-\frac{\hslash}{i}\right)\rho_{\hslash}\left(\widehat{\frac{\partial f}{\partial \xi_j}} \right),\ j=1,\cdots,n.
			\end{align*}
		\item [(2)] For $\hslash>0$, $f_1,f_2\in \mathcal{S}(T^*{\mathbb{R}}^n \times {\mathbb{R}};\mathrm{End}(V\oplus W))$,
		\begin{align*}
			\rho_{\hslash}(\hat{f_1})\rho_{\hslash}(\hat{f_2})=\rho_{\hslash}(\hat{f_1} \ast_H \hat{f_2}).
		\end{align*}
		\item [(3)] For $f\in \mathcal{S}(T^*{\mathbb{R}}^n \times {\mathbb{R}};\mathrm{End}(V\oplus W))$, $\forall g\in G$,
	    \begin{align*}
	    	\mathrm{Tr}(\rho_{\hslash}(\hat{f})\hat{g})=\frac{1}{{\hslash}^n(2\pi)^n}\int_{\mathbb{R}^n}\mathrm{tr}\left[\left(\begin{matrix}
	    		g^V & 0\\
	    		0 & g^W\\
	    	\end{matrix} \right)\hat{f}(x,\frac{gx-x}{\hslash},\hslash)\right] dx.
	    \end{align*}
		\end{itemize}
	\end{proposition}
	\begin{proof}
    \begin{itemize}
    	\item [(1)] We only need to verify that $\forall f\in \mathcal{S}(T^*{\mathbb{R}}^n \times {\mathbb{R}})$, $\forall j\in \left\{1,\cdots,n \right\}$,
    	\begin{align*}
    		\delta_{2j-1}(\rho_{\hslash}(\hat{f}))=\rho_{\hslash}\left(\widehat{\frac{\partial f}{\partial x_j}} \right),\ \delta_{2j}(\rho_{\hslash}(\hat{f}))=\left(-\frac{\hslash}{i}\right)\rho_{\hslash}\left(\widehat{\frac{\partial f}{\partial \xi_j}} \right).
    	\end{align*}
    Since $\hat{f}$ is the Fourier transform of $f$ with respect to the second variable, by direct computation,  $\forall \phi \in L^2(\mathbb{R}^n)$, $\forall x\in \mathbb{R}^n$,
    \begin{align*}
    	\delta_{2j-1}(\rho_{\hslash}(\hat{f}))\phi(x)&=(\partial_{x_j}\rho_{\hslash}(\hat{f})-\rho_{\hslash}(\hat{f})\partial_{x_j})\phi(x)\\
    	&=\frac{1}{{(2\pi)}^{n}}\int_{{\mathbb{R}^{n}}} \partial_{x_j}\left(\hat{f}(x,y,\hslash)\phi(x+\hslash y)\right) dy-\frac{1}{{(2\pi)}^{n}}\int_{{\mathbb{R}^{n}}} \hat{f}(x,y,\hslash)(\partial_{x_j}\phi)(x+\hslash y) dy\\
    	&=\frac{1}{{(2\pi)}^{n}}\int_{{\mathbb{R}^{n}}}\left(\widehat{\frac{\partial f}{\partial x_j}}(x,y,\hslash)\phi(x+\hslash y)\right) dy\\
    	&=\rho_{\hslash}\left(\widehat{\frac{\partial f}{\partial x_j}} \right)\phi(x),\\
    	\delta_{2j}(\rho_{\hslash}(\hat{f}))\phi(x)&=( x_j\rho_{\hslash}(\hat{f})-\rho_{\hslash}(\hat{f}) x_j)\phi(x)\\
    	&=\frac{1}{{(2\pi)}^{n}}(-\hslash)\int_{{\mathbb{R}^{n}}} y_j\hat{f}(x,y,\hslash)\phi(x+\hslash y)dy\\
    	&=\left(-\frac{\hslash}{i}\right)\rho_{\hslash}\left(\widehat{\frac{\partial f}{\partial \xi_j}} \right)\phi(x).
    \end{align*}
    \item [(2)] 
    The proof follows from direct computation. Hence we omit it.
    
    \item [(3)] For $\phi \in L^2(\mathbb{R}^n; V\oplus W)$ and $x\in \mathbb{R}^n$,
    \begin{align}
    	\label{s7}
    	(\rho_{\hslash}(\hat{f})\hat{g})\phi(x)&=\frac{1}{(2\pi)^n}\int_{\mathbb{R}^n}\hat{f}(x,y,\hslash)(\hat{g}.\phi)(x+\hslash y)dy\notag\\
    	&=\frac{1}{(2\pi)^n}\int_{\mathbb{R}^n}\hat{f}(x,y,\hslash)\left(\begin{matrix}
    		g^V & 0\\
    		0 & g^W\\
    	\end{matrix} \right)\phi(g^{-1}x+\hslash g^{-1}y)dy\notag\\
    	&=\frac{\mathrm{det}(g)}{{\hslash}^n(2\pi)^n}\int_{\mathbb{R}^n}\hat{f}(x,\frac{g\overline{y}-x}{\hslash},\hslash)\left(\begin{matrix}
    		g^V & 0\\
    		0 & g^W\\
    	\end{matrix} \right)\phi(\overline{y})d\overline{y}.
    \end{align}
    Since $g\in G$ and $G \leq SO(n)$,
    \begin{align*}
    	(\ref{s7})=\frac{1}{{\hslash}^n(2\pi)^n}\int_{\mathbb{R}^n}\hat{f}(x,\frac{g\overline{y}-x}{\hslash},\hslash)\left(\begin{matrix}
    		g^V & 0\\
    		0 & g^W\\
    	\end{matrix} \right)\phi(\overline{y})d\overline{y}.
    \end{align*}
    Moreover, as $\hat{f}$ is smooth, 
    \begin{align*}
    	\mathrm{Tr}(\rho_{\hslash}(\hat{f})\hat{g})=\frac{1}{{\hslash}^n(2\pi)^n}\int_{\mathbb{R}^n}\mathrm{tr}\left[\left(\begin{matrix}
    		g^V & 0\\
    		0 & g^W\\
    	\end{matrix} \right)\hat{f}(x,\frac{gx-x}{\hslash},\hslash)\right] dx.
    \end{align*}
    \end{itemize}		
	\end{proof}

	\begin{lemma}
		\label{lemma1}
		For $f_0,\cdots,f_{2n_g} \in \mathcal{S}(T^*{\mathbb{R}}^n \times {\mathbb{R}};\mathrm{End}(V\oplus W))$, if $\mathrm{dim}\left({\mathbb{R}^n}\right)^g=n_g>0$, then
		\begin{align*}
			\lim_{\hslash \to 0}\omega_g(\rho_{\hslash}(\hat{f_0}),\cdots,\rho_{\hslash}(\hat{f}_{2n_g}))= \epsilon_g(\rho_{0}(\hat{f_0}),\cdots,\rho_{0}(\hat{f}_{2n_g})).
		\end{align*}
	\end{lemma}
    \begin{proof}
    	The idea of proof is based on (\cite{ENN2}; Lemma 5.14).
    	By definition,
    	\begin{align*}
    		\omega_g(\rho_{\hslash}(\hat{f_0}),\cdots,\rho_{\hslash}(\hat{f}_{2n_g}))&= \frac{(-1)^{n_g}}{n_g!}\sum_{\sigma \in S_{2n_g}}{\mathrm{sgn}(\sigma)\mathrm{Tr}\left(\hat{g}\rho_{\hslash}(\hat{f_0}) \delta_{\sigma(1)}(\rho_{\hslash}(\hat{f}_{1})) \cdots \delta_{\sigma(2n_g)}(\rho_{\hslash}(\hat{f}_{2n_g})) \right)}.
    	\end{align*}
       We only need to analyze the term ${\mathrm{Tr}\left(\hat{g}\rho_{\hslash}(\hat{f_0}) \delta_{1}(\rho_{\hslash}(\hat{f}_{1})) \cdots \delta_{2n_g}(\rho_{\hslash}(\hat{f}_{2n_g})) \right)}$. By Proposition \ref{prop3}, for $\hslash>0$, 
       \begin{align*}
       	&{\mathrm{Tr}\left(\hat{g}\rho_{\hslash}(\hat{f_0}) \delta_{1}(\rho_{\hslash}(\hat{f}_{1})) \cdots \delta_{2n_g}(\rho_{\hslash}(\hat{f}_{2n_g})) \right)}\\
       	&=\left(-\frac{\hslash}{i}\right)^{n_g}\mathrm{Tr}\left(\hat{g}
       	\rho_{\hslash}(\widehat{f_0})\rho_{\hslash}\left(\widehat{\frac{\partial f_1}{\partial x_1}} \right) \cdots \rho_{\hslash}\left(\widehat{\frac{\partial f_{2n_g}}{\partial \xi_{n_g}}} \right)\right)\\
       	&=\left(-\frac{\hslash}{i}\right)^{n_g}\mathrm{Tr}\left(\hat{g}\rho_{\hslash}\left(\widehat{f_0}\ast_H \widehat{\frac{\partial f_1}{\partial x_1}}  \ast_H \cdots \ast_H  \widehat{\frac{\partial f_{2n_g}}{\partial \xi_{n_g}}}  \right)	\right)\\
       	&=\frac{i^{n_g}}{({2\pi })^n{\hslash}^{n-n_g}}
       	\int_{\mathbb{R}^n}\mathrm{tr}\left[\left(\begin{matrix}
       		g^V & 0\\
       		0 & g^W\\
       	\end{matrix} \right)\left(\widehat{f_0}\ast_H \widehat{\frac{\partial f_1}{\partial x_1}}  \ast_H \cdots \ast_H  \widehat{\frac{\partial f_{2n_g}}{\partial \xi_{n_g}}}  \right)(x,\frac{gx-x}{\hslash},\hslash) \right]dx\\
       	&=\frac{i^{n_g}}{({2\pi })^n{\hslash}^{n-n_g}}
       	\int_{(\mathbb{R}^n)^{g}}\int_{{\mathcal{N}{(\mathbb{R}^n)}^g}}\mathrm{tr}\left[\left(\begin{matrix}
       		g^V & 0\\
       		0 & g^W\\
       	\end{matrix} \right)\left(\widehat{f_0}\ast_H \widehat{\frac{\partial f_1}{\partial x_1}}  \ast_H \cdots \ast_H  \widehat{\frac{\partial f_{2n_g}}{\partial \xi_{n_g}}}  \right)((x_0,v),(0,\frac{gv-v}{\hslash}),\hslash) \right] dvdx_0,
       	\end{align*}
       	by variable substitution $\frac{v}{\hslash}=w$,
       	\begin{align*}
       		&=\frac{i^{n_g}}{({2\pi })^n}
       		\int_{(\mathbb{R}^n)^{g}}\int_{{\mathcal{N}{(\mathbb{R}^n)}^g}}\mathrm{tr}\left[\left(\begin{matrix}
       			g^V & 0\\
       			0 & g^W\\
       		\end{matrix} \right)\left(\widehat{f_0}\ast_H \widehat{\frac{\partial f_1}{\partial x_1}}  \ast_H \cdots \ast_H  \widehat{\frac{\partial f_{2n_g}}{\partial \xi_{n_g}}}  \right)((x_0,\hslash w),(0,{gw-w}),\hslash) \right]dwd{x_0}\\
       	&\stackrel{\hslash \to 0}{\longrightarrow} \frac{i^{n_g}}{({2\pi })^n}
       	\int_{(\mathbb{R}^n)^{g}}\int_{{\mathcal{N}{(\mathbb{R}^n)}^g}}\mathrm{tr}\left[\left(\begin{matrix}
       		g^V & 0\\
       		0 & g^W\\
       	\end{matrix} \right)\left(\widehat{f_0}\ast_H \widehat{\frac{\partial f_1}{\partial x_1}}  \ast_H \cdots \ast_H  \widehat{\frac{\partial f_{2n_g}}{\partial \xi_{n_g}}}  \right)((x_0,0),(0,{gw-w}),0) \right]dwd{x_0} \\
       &=\frac{i^{n_g}}{({2\pi })^n}
       \int_{(\mathbb{R}^n)^{g}}\int_{{\mathcal{N}{(\mathbb{R}^n)}^g}}\mathrm{tr}\left[\left(\begin{matrix}
       	g^V & 0\\
       	0 & g^W\\
       \end{matrix} \right)\left(\widehat{f_0\frac{\partial f_1}{\partial x_1}\cdots \frac{\partial f_{2n_g}}{\partial x_{2n_g}}} \right)((x_0,0),(0,{gw-w}),0) \right] dwdx_0\\
       	&=\frac{i^{n_g}}{({2\pi })^n}
       	\int_{(\mathbb{R}^n)^{g}}\int_{{\mathcal{N}{(\mathbb{R}^n)}^g}}\int_{\mathbb{R}^n}\mathrm{tr}\left[\left(\begin{matrix}
       		g^V & 0\\
       		0 & g^W\\
       	\end{matrix} \right)\left(f_0\frac{\partial f_1}{\partial x_1}\cdots \frac{\partial f_{2n_g}}{\partial x_{2n_g}} \right)((x_0,0),\xi,0)e^{-i\langle \xi, (0,{gw-w}) \rangle} \right]d\xi dwdx_0,
       \end{align*}
       by variable substitution
       $\xi=(\xi_1,\xi_2)$, where $\xi_1\in (\mathbb{R}^n)^{g}$ and $\xi_2 \in {\mathcal{N}{(\mathbb{R}^n)}^g}$,
       	\begin{align*}
       	&=\frac{i^{n_g}}{({2\pi })^n}
       	\int_{(\mathbb{R}^n)^{g}}\int_{{\mathcal{N}{(\mathbb{R}^n)}^g}}\int_{(\mathbb{R}^n)^{g}}\int_{{\mathcal{N}{(\mathbb{R}^n)}^g}}\mathrm{tr}\left[\left(\begin{matrix}
       		g^V & 0\\
       		0 & g^W\\
       	\end{matrix} \right) \left(f_0\frac{\partial f_1}{\partial x_1}\cdots \frac{\partial f_{2n_g}}{\partial x_{2n_g}} \right)((x_0,0),(\xi_1,\xi_2),0)e^{-i\langle \xi_2, {gw-w} \rangle}\right]
       	d\xi_2d\xi_1 dw dx_0,
       \end{align*}
   by variable substitution
   $u=gw-w$, denoted by the Jacobian determinant
   \begin{align*}
   	\mathcal{J}=\frac{1}{\mathrm{det}(g-1)},
   \end{align*} 
       \begin{align*}
       	&=\frac{i^{n_g}}{({2\pi })^n}
       	\int_{(\mathbb{R}^n)^{g}}\int_{{\mathcal{N}{(\mathbb{R}^n)}^g}}\int_{(\mathbb{R}^n)^{g}}\int_{{\mathcal{N}{(\mathbb{R}^n)}^g}}\mathrm{tr}\left[\left(\begin{matrix}
       		g^V & 0\\
       		0 & g^W\\
       	\end{matrix} \right) \left(f_0\frac{\partial f_1}{\partial x_1}\cdots \frac{\partial f_{2n_g}}{\partial x_{2n_g}} \right)((x_0,0),(\xi_1,\xi_2),0)e^{-i\langle \xi_2, u \rangle}\right]\mathcal{J}
       	d\xi_2d\xi_1 du dx_0\\
       	&=\frac{i^{n_g}}{({2\pi })^{n_g}}
       	\int_{(\mathbb{R}^n)^{g}}\int_{(\mathbb{R}^n)^{g}}\mathrm{tr}\left[\left(\begin{matrix}
       		g^V & 0\\
       		0 & g^W\\
       	\end{matrix} \right) \left(f_0\frac{\partial f_1}{\partial x_1}\cdots \frac{\partial f_{2n_g}}{\partial x_{2n_g}} \right)((x_0,0),(\xi_1,0),0)\right]\mathcal{J}
       	d\xi_1dx_0\\
       	&=\frac{i^{n_g}}{({2\pi })^{n_g}}
       	\int_{T^*(\mathbb{R}^n)^{g} \times \{0\}}\mathrm{tr}\left[\left(\begin{matrix}
       		g^V & 0\\
       		0 & g^W\\
       	\end{matrix} \right) \left(f_0\frac{\partial f_1}{\partial x_1}\cdots \frac{\partial f_{2n_g}}{\partial x_{2n_g}} \right)(x_0,\xi_1,0)\right]\mathcal{J}
       	d\xi_1dx_0,
       	\end{align*}
   where the penultimate equality follows from a technique involving the inversion theorem, i.e., $\forall f\in \mathcal{S}({\mathbb{R}^{t}})$, $t\in \mathbb{N}$,
   \begin{align*}
   	 f(0)=(2\pi)^{-t}\int_{{\mathbb{R}^{t}}}\int_{{\mathbb{R}^{t}}}f(y_1)e^{-i\langle y_1,y_2 \rangle}d{y_1}d{y_2}.
   \end{align*}
   The calculation for other terms is similar. Thus,
   \begin{align*}
   	&\lim_{\hslash \to 0}\omega_g(\rho_{\hslash}(\hat{f_0}),\cdots,\rho_{\hslash}(\hat{f}_{2n_g}))\\
   	&=\frac{1}{(2\pi i)^{n_g}n_g!{\mathrm{det}(g-1)}}
   	\int_{T^*(\mathbb{R}^n)^{g} \times \{0\}}\mathrm{tr}\left[\left(\begin{matrix}
   		g^V & 0\\
   		0 & g^W\\
   	\end{matrix} \right) \left(f_0 d{f_1}\cdots d{ f_{2n_g}} \right)\right]\\
   	&= \epsilon_g(\rho_{0}(\hat{f_0}),\cdots,\rho_{0}(\hat{f}_{2n_g})).
   \end{align*}
    \end{proof}

   Based on a similar proof, we obtain the following lemma.
   \begin{lemma}
   	\label{lemma3}
   	For $f_0,\cdots,f_{2n_g} \in \mathcal{S}(T^*{\mathbb{R}}^n \times {\mathbb{R}};\mathrm{End}(V\oplus W))$, if ${(\mathbb{R}^n)}^g=\{0\}$, then
   	\begin{align*}
   		\lim_{\hslash \to 0}\omega_g(\rho_{\hslash}(\hat{f_0}),\cdots,\rho_{\hslash}(\hat{f}_{2n_g}))= \epsilon_g(\rho_{0}(\hat{f_0}),\cdots,\rho_{0}(\hat{f}_{2n_g})).
   	\end{align*}
   \end{lemma}

	Finally, we complete the discussion of the equivariant analytic index. The rest of proofs of Theorem \ref{th3} and Theorem \ref{th4} is given as follows.
	Continued with the mapping (\ref{map4}), by Lemma \ref{lemma1} and Lemma \ref{lemma3}, we obtain
	\begin{align}
		\label{map5}
		\mathrm{ind}_G(P):\  G &\to \mathbb{C}\\
		g &\mapsto
		\begin{aligned}[t]
			\mathrm{ind}_{(g)}(P)&=\mathrm{Tr}(g|_{L^2({\mathbb{R}}^n; V)}\mathrm{Ker}P)-\mathrm{Tr}(g|_{L^2({\mathbb{R}}^n; W)}\mathrm{Ker}P^*)\\
			&=\langle \left[\mathrm{Ker}P\right]-[\mathrm{Ker}P^*],[(2\pi i)^{n_g}(n_g)!\omega_g] \rangle\\
			&=\langle \left[e_{1}\right]-\left[\left(\begin{matrix}
				0 & 0\\
				0 & 1\\
			\end{matrix}
			\right)\right],[(2\pi i)^{n_g}(n_g)!\omega_g] \rangle\\
			&=\langle \left[e^{\infty}_{1}\right]-\left[\left(\begin{matrix}
				0 & 0\\
				0 & 1\\
			\end{matrix}
			\right)\right],[(2\pi i)^{n_g}(n_g)!\omega_g] \rangle\\
			&=\langle \left[e^{\infty}_{\hslash}\right]-\left[\left(\begin{matrix}
				0 & 0\\
				0 & 1\\
			\end{matrix}
			\right)\right],[(2\pi i)^{n_g}(n_g)!\omega_g] \rangle\ \mathrm{for\  all\  \hslash>0}\\
			&=\langle \left[e^{\infty}_{0}\right]-\left[\left(\begin{matrix}
				0 & 0\\
				0 & 1\\
			\end{matrix}
			\right)\right],[(2\pi i)^{n_g}(n_g)!\epsilon_g] \rangle\\
			&=\langle \left[e_{0}\right]-\left[\left(\begin{matrix}
				0 & 0\\
				0 & 1\\
			\end{matrix}
			\right)\right],[(2\pi i)^{n_g}(n_g)!\epsilon_g] \rangle\\
 	 &=\frac{1}{(2\pi i)^{n_g}{n_g}!\mathrm{det}(g-1) }\int_{T^*{(\mathbb{R}^n)}^g} \mathrm{tr}\left[\left(\begin{matrix}
		g^V & 0\\
		0 & g^W\\
	\end{matrix}\right) \hat{e_a}(d\hat{e_a})^{2n_g}\right].\notag
    \end{aligned}
	\end{align}
    Here the last equality can be verified as follows (\cite{ENN2}; Page 530). Suppose $P$ is of order $m>0$. Let $A$ be the subalgebra of $\left(C_0(T^*{\mathbb{R}}^n;\mathrm{End}(V\oplus W))\right)^g$ consisting of symbols of order $\leq -m$. Since $A$ is stable under the holomorphic functional calculus in $\left(C_0(T^*{\mathbb{R}}^n;\mathrm{End}(V\oplus W))\right)^g$ and $\hat{e_a}\in A$ (\cite{ENN2}; Section 3), we obtain
    \begin{align*}
    	\left[e^{\infty}_{0}\right]-\left[\left(\begin{matrix}
    		0 & 0\\
    		0 & 1\\
    	\end{matrix}
    	\right)\right]=\left[e_{a}\right]-\left[\left(\begin{matrix}
    		0 & 0\\
    		0 & 1\\
    	\end{matrix}
    	\right)\right] \in K_0(A).
    \end{align*}
    Furthermore, since $\epsilon_g$ is defined on $A$, the last equality above is derived. 
   \section{Example of Bott-Dirac operator}

   Denote by $\mathrm{Cliff}_{\mathbb{C}}(\mathbb{R}^{2n})$ the complex Clifford algebra of the  $2n$-dimensional Euclidean space $\mathbb{R}^{2n}$, that is, the universal unital complex algebra containing $\mathbb{R}^{2n}$ as a real subspace  subject to the relations 
   \begin{align*}
   	xx=|x|^2
   \end{align*}
   for all $x\in \mathbb{R}^{2n}$. It is a 
   $\mathbb{Z}_2$-graded algebra with each $e\in \mathbb{R}^{2n}$ having grading degree one. Decompose as 
   \begin{align*}
   	 \mathrm{Cliff}_{\mathbb{C}}(\mathbb{R}^{2n})=(\mathrm{Cliff}_{\mathbb{C}}(\mathbb{R}^{2n}))_0\oplus (\mathrm{Cliff}_{\mathbb{C}}(\mathbb{R}^{2n}))_1,
   \end{align*}
   where $(\mathrm{Cliff}_{\mathbb{C}}(\mathbb{R}^{2n}))_0$ and  $(\mathrm{Cliff}_{\mathbb{C}}(\mathbb{R}^{2n}))_1$ denote the even and odd elements respectively.
   Fix an orthonormal basis $\{e_1,\cdots, e_{2n}\}$ of $\mathbb{R}^{2n}$, then the monomials $e_{i_1}\cdots e_{i_k}$ for $i_1<\cdots<i_k$ form a linear basis of $\mathrm{Cliff}_{\mathbb{C}}(\mathbb{R}^{2n})$. Define a Hermitian inner product on $\mathrm{Cliff}_{\mathbb{C}}(\mathbb{R}^{2n})$ by deeming these monomials to be orthonormal, which does not depend on the choice of basis.
   Let $L^2(\mathbb{R}^{2n};\mathrm{Cliff}_{\mathbb{C}}(\mathbb{R}^{2n}))$ denote the $\mathbb{Z}_2$-graded Hilbert space of square integrable functions from $\mathbb{R}^{2n}$ to $\mathrm{Cliff}_{\mathbb{C}}(\mathbb{R}^{2n})$, with $\mathbb{Z}_2$-grading induced from $\mathrm{Cliff}_{\mathbb{C}}(\mathbb{R}^{2n})$ and let $\mathcal{S}(\mathbb{R}^{2n};\mathrm{Cliff}_{\mathbb{C}}(\mathbb{R}^{2n}))$ denote the dense subspace of Schwartz-class functions from $\mathbb{R}^{2n}$ to $\mathrm{Cliff}_{\mathbb{C}}(\mathbb{R}^{2n})$.
   
   Denote by $x_1,\cdots,x_{2n}: \mathbb{R}^{2n} \to \mathbb{R}$ the corresponding coordinates.
   For $i\in \{1,\cdots, 2n\}$, define the Clifford multiplication operators $\hat{c}(e_i)$ and $c(e_i)$  on $\mathrm{Cliff}_{\mathbb{C}}(\mathbb{R}^{2n})$ by the formulas
   \begin{align*}
   	 \hat{c}(e_i)&: w \mapsto (-1)^{\mathrm{deg}(w)}we_i,\\
   	 c(e_i)&: w \mapsto e_iw,
   \end{align*}
   where $w$ is a homogeneous element. The Dirac operator $D$ and Bott (resp. Clifford) operator $C$ are unbounded operators on $L^2(\mathbb{R}^{2n};\mathrm{Cliff}_{\mathbb{C}}(\mathbb{R}^{2n}))$ given by
   \begin{align*}
   	 D&=\sum_{i=1}^{2n}\hat{c}(e_i)\frac{\partial}{\partial x_i},\\
   	 C&=\sum_{i=1}^{2n}{c}(e_i){x_i},
   \end{align*}
   with domain $\mathcal{S}(\mathbb{R}^{2n};\mathrm{Cliff}_{\mathbb{C}}(\mathbb{R}^{2n}))$. Here by abuse of notation, $\hat{c}(e_i)$ and $c(e_i)$ denote the induced operators on $L^2(\mathbb{R}^{2n};\mathrm{Cliff}_{\mathbb{C}}(\mathbb{R}^{2n}))$, while $x_i$ denotes the multiplication operator by the coordinate function $x_i$ on $L^2(\mathbb{R}^{2n};\mathrm{Cliff}_{\mathbb{C}}(\mathbb{R}^{2n}))$, $i\in \{1,\cdots, 2n\}$. 
   
   The Bott-Dirac operator is the unbounded operator on $L^2(\mathbb{R}^{2n};\mathrm{Cliff}_{\mathbb{C}}(\mathbb{R}^{2n}))$
   \begin{align*}
   	 B=D+C
   \end{align*}
   with domain $\mathcal{S}(\mathbb{R}^{2n};\mathrm{Cliff}_{\mathbb{C}}(\mathbb{R}^{2n}))$. Note that $B$ is an odd, essentially self-adjoint operator and has compact resolvent. 
   By considering $SO(2n)$ acting on $\mathbb{R}^{2n}$ by isometry, $\mathrm{Cliff}_{\mathbb{C}}(\mathbb{R}^{2n})$ carries a natural $SO(2n)$-action by diagnoal, which preserves the $\mathbb{Z}_2$-grading. Furthermore, $L^2(\mathbb{R}^{2n};\mathrm{Cliff}_{\mathbb{C}}(\mathbb{R}^{2n}))$ is an $SO(2n)$-Hilbert space with induced actions.
   
   Firstly we show that $B$ is an $SO(2n)$-invariant elliptic $\Psi$DO of positive order. For a fixed $g\in SO(2n)$,
   \begin{align*}
     gDg^{-1}=\sum_{i=1}^{2n}\hat{c}(g.e_i)\frac{\partial}{\partial (g^{-1}.x)_i},
   \end{align*} 
   which shows that $gDg^{-1}$ is the local expression of $D$ under another orthonormal basis $\{g.e_1\cdots,g.e_{2n}\}$ of $\mathbb{R}^{2n}$. Thus $D$ is $SO(2n)$-invariant and similarly, $B$ is also $SO(2n)$-invariant. As $B$ is an odd and essentially self-adjoint operator, denote by 
   \begin{align*}
   	\left(\begin{matrix}
   		0 & a^*\\
   		a & 0
   	\end{matrix}\right)
   \end{align*}
   the symbol of $B$. So 
   \begin{align*}
   	 \left(\begin{matrix}
   	 	0 & a^*\\
   	 	a & 0
   	 \end{matrix}\right)=\sum_{j=1}^{2n}(\hat{c}(e_j){i \xi_j}+{c}(e_k){x_j}).
   \end{align*}
   After taking square of both sides of the above formula, we obtain
   \begin{align*}
   	 \left(\begin{matrix}
   	 	a^*a & 0\\
   	 	0 & aa^*
   	 \end{matrix}\right)=|z|^2I,
   \end{align*}
   where $z=(x,\xi)\in {\mathbb{R}^{2n}\times \mathbb{R}^{2n}}$, thus by Definitions \ref{def5} and \ref{def4}, $a$ is elliptic and has order 1.  
   
   Then we calculate the equivariant index of $B$ by formulas of Theorem \ref{th3} and Theorem \ref{th4}.

   To begin with, we focus on the case when $n=1$. 
   For an orthonormal basis $\{e_1,e_2\}$ of $\mathbb{R}^{2}$, $\{1,e_1e_2,e_1,e_2\}$ forms a basis of $\mathrm{Cliff}_{\mathbb{C}}(\mathbb{R}^{2})=(\mathrm{Cliff}_{\mathbb{C}}(\mathbb{R}^{2}))_0\oplus (\mathrm{Cliff}_{\mathbb{C}}(\mathbb{R}^{2}))_1$ with respect to the grading. Under the basis, the Clifford multiplication operators are represented by
   \begin{align*}
   	 \hat{c}(e_1)&= \left(\begin{matrix}
   	 	0 &0 &-1 &0\\
   	 	0 &0 &0 &1\\
   	 	1 &0 &0 &0\\
   	 	0 &-1 &0 &0
   	 \end{matrix}\right),\\
    \hat{c}(e_2)&= \left(\begin{matrix}
    	0 &0 &0 &-1\\
    	0 &0 &-1 &0\\
    	0 &1 &0 &0\\
    	1 &0 &0 &0
    \end{matrix}\right),\\
   	{c}(e_1)&= \left(\begin{matrix}
   		0 &0 &1 &0\\
   		0 &0 &0 &1\\
   		1 &0 &0 &0\\
   		0 &1 &0 &0
   	\end{matrix}\right),\\
    {c}(e_2)&= \left(\begin{matrix}
   	0 &0 &0 &1\\
   	0 &0 &-1 &0\\
   	0 &-1 &0 &0\\
   	1 &0 &0 &0
   \end{matrix}\right).
   \end{align*}
   Therefore, the symbol of $B$ is given by
   \begin{align*}
   	\left(\begin{matrix}
   		0& a^*\\
   		a& 0
   	\end{matrix}\right)\  \mathrm{with}\  a=\left(\begin{matrix}
   	x_1+i\xi_1 & -x_2+i\xi_2 \\
   	x_2+i\xi_2 & x_1-i\xi_1
   \end{matrix}\right),
   \end{align*}
   where $(x_1,x_2,\xi_1,\xi_2)\in T^*\mathbb{R}^{2}$ and $x_1,x_2,\xi_1,\xi_2\in \mathbb{R}$.
   Note that $B: \mathcal{S}(\mathbb{R}^{2};(\mathrm{Cliff}_{\mathbb{C}}(\mathbb{R}^{2}))_0) \to \mathcal{S}(\mathbb{R}^{2};(\mathrm{Cliff}_{\mathbb{C}}(\mathbb{R}^{2}))_1)$. For
   \begin{align*}
   	 g=\left(\begin{matrix}
   	 	\mathrm{cos}\theta & \mathrm{-sin}\theta\\
   	 	\mathrm{sin}\theta & \mathrm{cos}\theta
   	 \end{matrix}\right)\in SO(2), \ \mathrm{where}\ \theta \in [0,2\pi],
   \end{align*}
   under the basis $\{1,e_1e_2\}$ of $(\mathrm{Cliff}_{\mathbb{C}}(\mathbb{R}^{2}))_0$, the $g$-action is given by
   \begin{align*}
   	 g^{(\mathrm{Cliff}_{\mathbb{C}}(\mathbb{R}^{2}))_0}= \left(\begin{matrix}
   	 	1 & 0 \\
   	 	0 & 1
   	 \end{matrix}\right),
   \end{align*}
   while under the basis $\{e_1,e_2\}$ of $(\mathrm{Cliff}_{\mathbb{C}}(\mathbb{R}^{2}))_1$, the $g$-action is given by
   \begin{align*} g^{(\mathrm{Cliff}_{\mathbb{C}}(\mathbb{R}^{2}))_1}=
   	\left(\begin{matrix}
   		\mathrm{cos}\theta & \mathrm{-sin}\theta\\
   		\mathrm{sin}\theta & \mathrm{cos}\theta
   	\end{matrix}\right).
   \end{align*}
   When $g=1\in SO(2)$, the whole $\mathbb{R}^{2}$ is fixed by $g$. So by Theorem \ref{th3} or Theorem \ref{th1},
   \begin{align*}
   	 \mathrm{ind}_{(1)}(B)= \frac{1}{(2\pi i)^22!}\int_{T^*\mathbb{R}^2}^{} \mathrm{tr}(\hat{e_a}(d\hat{e_a})^{4}).
   \end{align*} 
   Since 
   \begin{align*}
   	\hat{e_a}= \frac{1}{1+{x_1}^2+{x_2}^2+{\xi_1}^2+{\xi_2}^2}\left(\begin{matrix}
   		I_2 & a^*\\
   		a & -I_2
   	\end{matrix}\right),
   \end{align*}
   by direct computation, 
   \begin{align*}
   	 \mathrm{ind}_{(1)}(B)&=\frac{1}{(2\pi i)^22!}\int_{T^*\mathbb{R}^2}^{}\frac{-4\times 4!dx_1dx_2d\xi_1d\xi_2}{\left({1+{x_1}^2+{x_2}^2+{\xi_1}^2+{\xi_2}^2}\right)^5}\\
   	 &=1.
   \end{align*}
   When $g=\left(\begin{matrix}
   	\mathrm{cos}\theta & \mathrm{-sin}\theta\\
   	\mathrm{sin}\theta & \mathrm{cos}\theta
   \end{matrix}\right)\neq1\in SO(2)$, $\mathrm{cos}\theta\neq 1$ and only the origin of $\mathbb{R}^{2}$ fixed by $g$, by Theorem \ref{th4},
   \begin{align*}
   	 \mathrm{ind}_{(g)}(B)&= \frac{1}{\mathrm{det}(g-1) } \mathrm{tr}\left[\left(\begin{matrix}
   	 	g^{(\mathrm{Cliff}_{\mathbb{C}}(\mathbb{R}^{2}))_0} & 0\\
   	 	0 & g^{(\mathrm{Cliff}_{\mathbb{C}}(\mathbb{R}^{2}))_1}\\
   	 \end{matrix}
   	 \right) \hat{e_a}(0,0)\right]\\
   	 &=\frac{\mathrm{tr}\left(g^{(\mathrm{Cliff}_{\mathbb{C}}(\mathbb{R}^{2}))_0}\right) - \mathrm{tr}\left(g^{(\mathrm{Cliff}_{\mathbb{C}}(\mathbb{R}^{2}))_1} \right)}{\mathrm{det}(g-1)}\\
   	 &=\frac{2-2\mathrm{cos}\theta}{2-2\mathrm{cos}\theta}\\
   	 &=1.
   \end{align*}
   Therefore, the equivariant index of $B$ is 1 at each $g\in SO(2)$.

   As for the case when $n>1$, we proceed using the language of Kasparov KK-theory. Denote by $[B_k]\in KK_{SO(2k)}(\mathbb{C},\mathbb{C})$ the unbounded Kasparov cycle associated to the Bott-Dirac operator on $\mathbb{R}^{2k}$ where $k\in \mathbb{N}$. By the $*$-isomorphism
   \begin{align*}
   	 \mathrm{Cliff}_{\mathbb{C}}(\mathbb{R}^{2n})\cong \underbrace{\mathrm{Cliff}_{\mathbb{C}}(\mathbb{R}^{2}) \widehat{\otimes} \cdots \widehat{\otimes} \mathrm{Cliff}_{\mathbb{C}}(\mathbb{R}^{2})}_{n\ \mathrm{times}},
   \end{align*}
   of graded Hilbert spaces, where $\widehat{\otimes}$ denotes the graded tensor product, we obtain
   \begin{align*}
   	 B_{n}=\sum_{i=1}^{n} 1\widehat{\otimes}\cdots \widehat{\otimes}1 \widehat{\otimes} \underbrace{B_1}_{i\mathrm{th}\ \mathrm{place}} \widehat{\otimes}1 \widehat{\otimes} \cdots \widehat{\otimes} 1
   \end{align*}
   with respect to the decomposition. 
   Note that the decompositions of $\mathrm{Cliff}_{\mathbb{C}}(\mathbb{R}^{2n})$ and $B_n$ are also compatible with the restricted actions of $\prod_{i=1}^{n}SO(2)$,  it derives that
   \begin{align*}
   	[B_{n}]=\underbrace{[B_1]\times \cdots \times [B_1]}_{n\ \mathrm{times}} \in KK_{SO(2n)}(\mathbb{C},\mathbb{C}),
   \end{align*}
    where the right hand side denotes the Kasparov product in the sense of Section 4 of \cite{Kasparov} and we use the identification that an element of the representation ring $R(SO(2n))$ is uniquely determined by its restriction to the maximal torus $\prod_{i=1}^{n}SO(2)$. 
    From above discussion, $[B_1]=1\in KK_{SO(2)}(\mathbb{C},\mathbb{C})$, so we obtain $[B_n]=1\in KK_{SO(2n)}(\mathbb{C},\mathbb{C})$, which derives that the equivariant index of $2n$-dimensional Bott-Dirac operator is 1 at each $g\in SO(2n)$.

	\section{Appendix}
	\begin{proposition}
		\label{prop6}
		Define a cyclic $2n$-cocycle $\omega$ on $\mathcal{K}^{\infty}$ as follows \cite{ENN2}:
		\begin{align*}
			\omega(T_0,\cdots,T_{2n})=  \frac{(-1)^{n}}{n!}\sum_{\sigma \in S_{2n}}{\mathrm{sgn}(\sigma)\mathrm{Tr}\left(T_0 \delta_{\sigma(1)}(T_1) \cdots \delta_{\sigma(2n)}(T_{2n}) \right)},
		\end{align*}
	where
	\begin{align*}
		T_0,\cdots,T_{2n}\in \mathcal{K}^{\infty}.
	\end{align*}
	Then for any idempotent $T \in\mathcal{K}^{\infty}$, $\langle [T],[(2\pi i)^{n}n!\omega] \rangle=\mathrm{Tr}(T)$.
	\end{proposition}
	\begin{proof}
	 
	 Consider the case of $\mathbb{R}^n$. Denote by 
	 \begin{align*}
	 	\omega_{2k}(T_0,\cdots,T_{2k})=  \frac{(-1)^{k}}{k!}\sum_{\sigma \in S_{2k}}{\mathrm{sgn}(\sigma)\mathrm{Tr}\left(T_0 \delta_{\sigma(1)}(T_1) \cdots \delta_{\sigma(2k)}(T_{2k}) \right)},
	 \end{align*}
     where $k=1,\cdots,n$.
     Induction on $k$.
     Fix an idempotent $T \in\mathcal{K}^{\infty}$, observe that
	 $\langle T,[(2\pi i)\omega_2] \rangle=\mathrm{Tr}(T)$.
	 For $1\leq m<n$, assume that by induction
	 \begin{align*}
	 	\langle [T],[(2\pi i)^{m}m!\omega_{2m}] \rangle=\mathrm{Tr}(T),
	 \end{align*}
	 then we shall construct a $(2m+1)$-cochain $\psi_{2m+1}$ on $\mathcal{K}^{\infty}$ such that
	 \begin{align}
	 	\label{formula1}
	 	B\psi_{2m+1}=-\frac{2(2m+1)}{m+1}\sum \omega_{2m},\ 
	 	b\psi_{2m+1}=-\omega_{2m+2}.
	 \end{align}
     Here $\sum \omega_{2m}$ is a cyclic cocycle on $\mathcal{K}^{\infty}$ given below.

	Let us define a $(2m+1)$-cochain $\psi_{2m+1}$ on $\mathcal{K}^{\infty}$ by
	\begin{align*}
		\psi_{2m+1}(T_0,\cdots,T_{2m+1})=\frac{(-1)^{m+1}}{(m+1)!}\sum_{\sigma \in S_{2m+2}}{\mathrm{sgn}(\sigma)\mathrm{Tr}\left(\widetilde{\delta_{\sigma(1)}}(T_0) \delta_{\sigma(2)}(T_1) \cdots \delta_{\sigma(2m+2)}(T_{2m+1}) \right)},
	\end{align*}
    where $T_0,\cdots,T_{2m+1}\in \mathcal{K}^{\infty}$.
    Here the operators $\widetilde{\delta_{*}}$ are given by
    \begin{align*}
    	\widetilde{\delta_{\sigma(2j-1)}}(T)&=T\partial_{x_j},\\
    	\widetilde{\delta_{\sigma(2j)}}(T)&=Tx_j,
    \end{align*}
    where $T\in \mathcal{K}^{\infty}$ and $j=1,\cdots, m+1$, and the operators $\delta_{*}$ are given in Section \ref{3.4}.
    
    Firstly verify that $B\psi_{2m+1}=-\frac{2(2m+1)}{m+1}\sum \omega_{2m}$. In fact, $B\psi_{2m+1}=NB_0\psi_{2m+1}$, where
    \begin{align*}
    	B_0\psi_{2m+1}(T_0,\dots,T_{2m})&=\psi_{2m+1}(1,T_0,\dots,T_{2m})-(-1)^{2m+1}\psi_{2m+1}(T_0,\dots,T_{2m},1)\\
    	&=\frac{(-1)^{m+1}}{(m+1)!}\sum_{\sigma \in S_{2m+2}}{\mathrm{sgn}(\sigma)\mathrm{Tr}\left(\widetilde{\delta_{\sigma(1)}}(1) \delta_{\sigma(2)}(T_0) \cdots \delta_{\sigma(2m+2)}(T_{2m}) \right)}.
    \end{align*}
    We claim that $\lambda B_0\psi_{2m+1}=B_0\psi_{2m+1}$.
    Actually, 
    \begin{align}
    	\label{s8}
    	\lambda B_0\psi_{2m+1}(T_0,\dots,T_{2m})&=(-1)^{2m}B_0\psi_{2m+1}(T_{2m},T_0,\dots,T_{2m-1})\notag\\
    	&=\frac{(-1)^{m+1}}{(m+1)!}\sum_{\tau \in S_{2m+2}}{\mathrm{sgn}(\tau)\mathrm{Tr}\left(\widetilde{\delta_{\tau(1)}}(1) \delta_{\tau(2)}(T_{2m})\delta_{\tau(3)}(T_{0}) \cdots \delta_{\tau(2m+2)}(T_{2m-1}) \right)},
    \end{align}
    taking $\tau=(\sigma(1), \sigma(2m+2),\sigma(2),\cdots,\sigma(2m+1))$ for variable $\sigma \in S_{2m+2}$,
    then $\mathrm{det}(\tau)=\mathrm{det}(\sigma)$ and
    \begin{align*}
    	(\ref{s8})&=\frac{(-1)^{m+1}}{(m+1)!}\sum_{\sigma \in S_{2m+2}}{\mathrm{sgn}(\sigma)\mathrm{Tr}\left(\widetilde{\delta_{\sigma(1)}}(1) \delta_{\sigma(2m+2)}(T_{2m})\delta_{\sigma(2)}(T_{0}) \cdots \delta_{\sigma(2m+1)}(T_{2m-1}) \right)}.
    \end{align*}
    Since
    \begin{align*}
    	B_0\psi_{2m+1}(T_0,\dots,T_{2m})&=\frac{(-1)^{m+1}}{(m+1)!}\sum_{\sigma \in S_{2m+2}}{\mathrm{sgn}(\sigma)\mathrm{Tr}\left( \delta_{\sigma(2m+2)}(T_{2m})\widetilde{\delta_{\sigma(1)}}(1) \delta_{\sigma(2)}(T_0) \cdots \delta_{\sigma(2m+1)}(T_{2m-1}) \right)},
    \end{align*}
    we obtain
    \begin{align*}
    	(&\lambda B_0\psi_{2m+1}-B_0\psi_{2m+1})(T_0,\dots,T_{2m})\\
    	&=\frac{(-1)^{m+1}}{(m+1)!}\sum_{\sigma \in S_{2m+2}}{\mathrm{sgn}(\sigma)\mathrm{Tr}\left(({\delta_{\sigma(1)}} \delta_{\sigma(2m+2)})(T_{2m})\delta_{\sigma(2)}(T_{0}) \cdots \delta_{\sigma(2m+1)}(T_{2m-1}) \right)}\\
    	&=0,
    \end{align*}
    where the last equality follows from the coefficient $\mathrm{sgn}(\sigma)$ and $\sigma_i\sigma_j=\sigma_j\sigma_i$, for any $i,j\in \{1,\cdots,2m+2\}$. 
    
    Hence
    \begin{align*}
    	NB_0\psi_{2m+1}(T_0,\dots,T_{2m})&=(2m+1)B_0\psi_{2m+1}(T_0,\dots,T_{2m})\\
    	&=(2m+1)\frac{(-1)^{m+1}}{(m+1)!}\sum_{\sigma \in S_{2m+2}}{\mathrm{sgn}(\sigma)\mathrm{Tr}\left(\widetilde{\delta_{\sigma(1)}}(1) \delta_{\sigma(2)}(T_0) \cdots \delta_{\sigma(2m+2)}(T_{2m}) \right)}\\
    	&=\frac{2m+1}{m+1}\frac{(-1)^{m}}{m!}\sum_{\sigma \in S_{2m+2}}{\mathrm{sgn}(\sigma)\mathrm{Tr}\left(\delta_{\sigma(2)}(\widetilde{\delta_{\sigma(1)}(1)})T_0 \cdots \delta_{\sigma(2m+2)}(T_{2m}) \right)}.
    \end{align*}
    The last equality follows from
    \begin{align*}
    	\mathrm{Tr}\left(\delta_{\sigma(2)}\left(\widetilde{\delta_{\sigma(1)}(1)}T_0 \cdots \delta_{\sigma(2m+2)}(T_{2m})\right) \right)=0,
    \end{align*}
    where the left hand side can be decomposed by the derivation property of $\delta_{\sigma(2)}$
    and since  $\sigma_i\sigma_j=\sigma_j\sigma_i$, for any $i,j\in \{1,\cdots,2m+2\}$, 
    \begin{align*}
    	\sum_{\sigma \in S_{2m+2}}{\mathrm{sgn}(\sigma)\mathrm{Tr}\left(\widetilde{\delta_{\sigma(1)}(1)}T_0\delta_{\sigma(2)}\left(\delta_{\sigma(3)}(T_{1}) \cdots \delta_{\sigma(2m+2)}(T_{2m})\right) \right)}=0.
    \end{align*}
    By definition of $\widetilde{\delta_{*}}$, $\widetilde{\delta_{\sigma(1)}(1)}=\partial_{x_j}$ or $x_j$ for some $j$, thus $\delta_{\sigma(2)}(\widetilde{\delta_{\sigma(1)}(1)})$ vanishes unless
    $(\sigma(1), \sigma(2))$ is one pair of $(1,2),\cdots,(2m+1,2m+2)$. In this way, for a fixed $\sigma \in S_{2m+2}$, define an induced $\sigma_0 \in S_{2m}$ by
    \begin{align*}
    	\sigma=(\sigma(1),\sigma(2),\sigma_0(1),\cdots,\sigma_0(2m)),
    \end{align*}
    then 
    \begin{align*}
    	\mathrm{sgn}(\sigma)=\mathrm{sgn}(\sigma(1),\sigma(2))\mathrm{sgn}(\sigma_0).
    \end{align*}
    Hence
    \begin{align*}
    	NB_0\psi_{2m+1}(T_0,\dots,T_{2m})=-\frac{2(2m+1)}{m+1}\frac{(-1)^{m}}{m!}\sum_{k=1}^{m+1}\sum_{\sigma_0 \in S^{(k)}_{2m}}{\mathrm{sgn}(\sigma_0)\mathrm{Tr}\left(T_0\delta_{\sigma_0(1)}(T_{1}) \cdots \delta_{\sigma_0(2m)}(T_{2m}) \right)},
    \end{align*}
    where $S^{(k)}_{2m}$ is the symmetric group of the set
    \begin{align*}
    	\{1,2,\cdots,2m+1,2m+2\}
    \end{align*}
    subtracting $\{2k-1, 2k\}$.
    Note that on the right hand side, the term of $k=1$ corresponds to $\omega_{2m}$ with a constant coefficient. Denote by $\sum \omega_{2m}$ another cyclic cocycle on $\mathcal{K}^{\infty}$ given by
    \begin{align*}
    	\sum \omega_{2m}(T_0,\cdots,T_{2m})=\frac{(-1)^{m}}{m!}\sum_{k=1}^{m+1}\sum_{\sigma_0 \in S^{(k)}_{2m}}{\mathrm{sgn}(\sigma_0)\mathrm{Tr}\left(T_0\delta_{\sigma_0(1)}(T_{1}) \cdots \delta_{\sigma_0(2m)}(T_{2m}) \right)},
    \end{align*}
    where $T_0,\cdots,T_{2m} \in \mathcal{K}^{\infty}$,
    then we have proved that
    \begin{align*}
    	B\psi_{2m+1}=-\frac{2(2m+1)}{m+1}\sum \omega_{2m}.
    \end{align*}
    
    Next verify that $b\psi_{2m+1}=-\omega_{2m+2}$.
    Since 
    \begin{equation}
    	\begin{gathered}
    		\label{s9}
    		\begin{aligned}
    			b\psi_{2m+1}(T_0,\cdots,T_{2m+2})=\psi_{2m+1}(&T_0T_1,T_2,\cdots,T_{2m+2})-\psi_{2m+1}(T_0,T_1T_2,\cdots,T_{2m+2})\\
    			&+(-1)^{2m+1}\psi_{2m+1}(T_0,T_1,\cdots,T_{2m+1}T_{2m+2})\\
    			&+(-1)^{2m+2}\psi_{2m+1}(T_{2m+2}T_0,T_1,\cdots,T_{2m+1}),
    		\end{aligned}
    	\end{gathered}
    \end{equation}
    by definition,
    \begin{align*}
    	(\ref{s9})&=\frac{(-1)^{m+1}}{(m+1)!}\sum_{\sigma \in S_{2m+2}}\mathrm{sgn}(\sigma)\left[{\mathrm{Tr}\left(\widetilde{\delta_{\sigma(1)}}(T_0T_1) \delta_{\sigma(2)}(T_2) \cdots \delta_{\sigma(2m+2)}(T_{2m+2}) \right)}\right. \\ 
    	&\;+ \sum_{i=1}^{2m+1}{(-1)^i\mathrm{Tr}\left(\widetilde{\delta_{\sigma(1)}}(T_0) \delta_{\sigma(2)}(T_1) \cdots \delta_{\sigma(i)}(T_iT_{i+1})\cdots\delta_{\sigma(2m+2)}(T_{2m+2}) \right)}\\
    	&\;+\left. {\mathrm{Tr}\left(\widetilde{\delta_{\sigma(1)}}(T_{2m+2}T_0) \delta_{\sigma(2)}(T_1) \cdots \delta_{\sigma(2m+2)}(T_{2m+1}) \right)}\right].
    \end{align*}
    After cancelations, the remaining two terms can be combined into:
    \begin{align*}
    	(\ref{s9})&=\frac{(-1)^{m}}{(m+1)!}\sum_{\sigma \in S_{2m+2}}\mathrm{Tr}\left(T_0\delta_{\sigma(1)}(T_1) \delta_{\sigma(2)}(T_2) \cdots \delta_{\sigma(2m+2)}(T_{2m+2}) \right)\\
    	&=-\omega_{2m+2}(T_0,\cdots,T_{2m+2}).
    \end{align*}
    
    Finally, by the expression of $S$ (\cite{Connes}; Page 337), equality (\ref{formula1}) implies that
    \begin{align*}
    	S(\sum\omega_{2m})&=2\pi i(2m+1)(2m+2)bB^{-1}(\sum\omega_{2m})\\
    	&=2\pi i(m+1)^2\omega_{2m+2},
    \end{align*}  
    then by Proposition 14 of \cite{Connes}, we obtain
    \begin{align*}
    	\langle [T],[(2\pi i)^{m}m!(\sum\omega_{2m})] \rangle=\langle [T],[(2\pi i)^{m}{m}!(S\sum\omega_{2m})] \rangle =(m+1)\langle [T],[(2\pi i)^{m+1}{(m+1)}!\omega_{2m+2}] \rangle.
    \end{align*}
    By induction hypothesis,
    \begin{align*}
    	\langle [T],[(2\pi i)^{m}m!(\sum\omega_{2m})] \rangle=(m+1)\mathrm{Tr}(T),
    \end{align*}
    followed the conclusion for $m+1$.

	\end{proof}


\begin{thebibliography}{99}
		
		\bibitem{Nest}
		Nest, R, Tsygan, B. Algebraic index theorem. Comm Math Phys, 1995, 172: 223--262.
		
		\bibitem{Nest2}
		Nest, R, Tsygan, B. Algebraic index theorem for families. Adv Math, 1995, 113: 151--205.
		
		
		\bibitem{ENN2}
		Elliott, G A, Natsume, T, Nest, R. The Atiyah-Singer index theorem as passage to the classical limit in quantum mechanics. Commun Math Phys, 1996, 182: 505--533. 
		
		\bibitem{Sh}
		Shubin, M A. Pseudodifferential Operators and Spectral Theory. Berlin: Springer-Verlag, 1987, 175--202.
		
		
		\bibitem{ENN1}
		Elliott, G A, Natsume, T, Nest, R. The Heisenberg group and $K$-theory. $K$-Theory, 1993, 7: 409--428. 
		
		
		\bibitem{Kasparov}
		Kasparov, G G. Topological invariants of elliptic operators. {I}.
		{$K$}-homology. Izv Akad Nauk SSSR Ser Mat, 1975, 39: 796--838; translation in
		Math USSR-Izv, 1975, 9: 751--792.
		
		
		\bibitem{Connes}
		Connes, A. Noncommutative differential geometry. Inst Hautes \'{E}tudes Sci Publ Math, 1985, 62: 257--360.
		
		
		
		
		 
		
		
		
	\end{thebibliography}
\end{document}